\documentclass[reqno,12pt,a4paper]{amsart}

\usepackage{amssymb}
\usepackage{amsfonts}
\usepackage{latexsym}
\usepackage{amsthm}

\usepackage{graphicx}

\begin{document}


\renewcommand{\theequation}{\arabic{section}.\arabic{equation}}
\theoremstyle{plain}
\newtheorem{theorem}{\bf Theorem}[section]
\newtheorem{lemma}[theorem]{\bf Lemma}
\newtheorem{corollary}[theorem]{\bf Corollary}
\newtheorem{proposition}[theorem]{\bf Proposition}
\newtheorem{definition}[theorem]{\bf Definition}
\newtheorem{remark}[theorem]{\it Remark}

\def\a{\alpha}  \def\cA{{\mathcal A}}     \def\bA{{\bf A}}  \def\mA{{\mathscr A}}
\def\b{\beta}   \def\cB{{\mathcal B}}     \def\bB{{\bf B}}  \def\mB{{\mathscr B}}
\def\g{\gamma}  \def\cC{{\mathcal C}}     \def\bC{{\bf C}}  \def\mC{{\mathscr C}}
\def\G{\Gamma}  \def\cD{{\mathcal D}}     \def\bD{{\bf D}}  \def\mD{{\mathscr D}}
\def\d{\delta}  \def\cE{{\mathcal E}}     \def\bE{{\bf E}}  \def\mE{{\mathscr E}}
\def\D{\Delta}  \def\cF{{\mathcal F}}     \def\bF{{\bf F}}  \def\mF{{\mathscr F}}
\def\c{\chi}    \def\cG{{\mathcal G}}     \def\bG{{\bf G}}  \def\mG{{\mathscr G}}
\def\z{\zeta}   \def\cH{{\mathcal H}}     \def\bH{{\bf H}}  \def\mH{{\mathscr H}}
\def\e{\eta}    \def\cI{{\mathcal I}}     \def\bI{{\bf I}}  \def\mI{{\mathscr I}}
\def\p{\psi}    \def\cJ{{\mathcal J}}     \def\bJ{{\bf J}}  \def\mJ{{\mathscr J}}
\def\vT{\Theta} \def\cK{{\mathcal K}}     \def\bK{{\bf K}}  \def\mK{{\mathscr K}}
\def\k{\kappa}  \def\cL{{\mathcal L}}     \def\bL{{\bf L}}  \def\mL{{\mathscr L}}
\def\l{\lambda} \def\cM{{\mathcal M}}     \def\bM{{\bf M}}  \def\mM{{\mathscr M}}
\def\L{\Lambda} \def\cN{{\mathcal N}}     \def\bN{{\bf N}}  \def\mN{{\mathscr N}}
\def\m{\mu}     \def\cO{{\mathcal O}}     \def\bO{{\bf O}}  \def\mO{{\mathscr O}}
\def\n{\nu}     \def\cP{{\mathcal P}}     \def\bP{{\bf P}}  \def\mP{{\mathscr P}}
\def\r{\rho}    \def\cQ{{\mathcal Q}}     \def\bQ{{\bf Q}}  \def\mQ{{\mathscr Q}}
\def\s{\sigma}  \def\cR{{\mathcal R}}     \def\bR{{\bf R}}  \def\mR{{\mathscr R}}
\def\S{\Sigma}  \def\cS{{\mathcal S}}     \def\bS{{\bf S}}  \def\mS{{\mathscr S}}
\def\t{\tau}    \def\cT{{\mathcal T}}     \def\bT{{\bf T}}  \def\mT{{\mathscr T}}
\def\f{\phi}    \def\cU{{\mathcal U}}     \def\bU{{\bf U}}  \def\mU{{\mathscr U}}
\def\F{\Phi}    \def\cV{{\mathcal V}}     \def\bV{{\bf V}}  \def\mV{{\mathscr V}}
\def\P{\Psi}    \def\cW{{\mathcal W}}     \def\bW{{\bf W}}  \def\mW{{\mathscr W}}
\def\o{\omega}  \def\cX{{\mathcal X}}     \def\bX{{\bf X}}  \def\mX{{\mathscr X}}
\def\x{\xi}     \def\cY{{\mathcal Y}}     \def\bY{{\bf Y}}  \def\mY{{\mathscr Y}}
\def\X{\Xi}     \def\cZ{{\mathcal Z}}     \def\bZ{{\bf Z}}  \def\mZ{{\mathscr Z}}
\def\O{\Omega}

\def\ve{\varepsilon}   \def\vt{\vartheta}    \def\vp{\varphi}    \def\vk{\varkappa}

\def\Z{{\mathbb Z}}    \def\R{{\mathbb R}}   \def\C{{\mathbb C}}
\def\T{{\mathbb T}}    \def\N{{\mathbb N}}   \def\dD{{\mathbb D}}


\def\la{\leftarrow}              \def\ra{\rightarrow}            \def\Ra{\Rightarrow}
\def\ua{\uparrow}                \def\da{\downarrow}
\def\lra{\leftrightarrow}        \def\Lra{\Leftrightarrow}


\def\lt{\biggl}                  \def\rt{\biggr}
\def\ol{\overline}               \def\wt{\widetilde}


\let\ge\geqslant                 \let\le\leqslant
\def\lan{\langle}                \def\ran{\rangle}
\def\/{\over}                    \def\iy{\infty}
\def\sm{\setminus}               \def\es{\emptyset}
\def\ss{\subset}                 \def\ts{\times}
\def\pa{\partial}                \def\os{\oplus}
\def\om{\ominus}                 \def\ev{\equiv}
\def\iint{\int\!\!\!\int}        \def\iintt{\mathop{\int\!\!\int\!\!\dots\!\!\int}\limits}
\def\el2{\ell^{\,2}}             \def\1{1\!\!1}
\def\sh{\sharp}


\def\Area{\mathop{\mathrm{Area}}\nolimits}
\def\arg{\mathop{\mathrm{arg}}\nolimits}
\def\const{\mathop{\mathrm{const}}\nolimits}
\def\det{\mathop{\mathrm{det}}\nolimits}
\def\diag{\mathop{\mathrm{diag}}\nolimits}
\def\diam{\mathop{\mathrm{diam}}\nolimits}
\def\dim{\mathop{\mathrm{dim}}\nolimits}
\def\dist{\mathop{\mathrm{dist}}\nolimits}
\def\Im{\mathop{\mathrm{Im}}\nolimits}
\def\Iso{\mathop{\mathrm{Iso}}\nolimits}
\def\Ker{\mathop{\mathrm{Ker}}\nolimits}
\def\Lip{\mathop{\mathrm{Lip}}\nolimits}
\def\rank{\mathop{\mathrm{rank}}\limits}
\def\Ran{\mathop{\mathrm{Ran}}\nolimits}
\def\Re{\mathop{\mathrm{Re}}\nolimits}
\def\Res{\mathop{\mathrm{Res}}\nolimits}
\def\res{\mathop{\mathrm{res}}\limits}
\def\sign{\mathop{\mathrm{sign}}\nolimits}
\def\span{\mathop{\mathrm{span}}\nolimits}
\def\supp{\mathop{\mathrm{supp}}\nolimits}
\def\Tr{\mathop{\mathrm{Tr}}\nolimits}


\newcommand\nh[2]{\widehat{#1}\vphantom{#1}^{(#2)}}
\def\dia{\diamond}

\def\Oplus{\bigoplus\nolimits}


\title[{Matrix-valued Weyl-Titchmarsh functions on [0,1]}]
{Weyl-Titchmarsh functions of vector-valued Sturm-Liouville operators on the unit interval}
\author[Dmitry Chelkak]{Dmitry Chelkak}

\author[Evgeny Korotyaev]{Evgeny Korotyaev}

\date{September 4, 2008}



\thanks{\textsc{Dept. of Math. Analysis, St.~Petersburg State University.
Universitetskij pr. 28, Staryj Petergof, 198504 St.~Petersburg, Russia.} \quad Partially
supported by the Foundation of the President of the Russian Federation (grants no.
MK-4306.2008.1 and NSh-2409.2008.1).}

\thanks{\textsc{School of Math., Cardiff University. Senghennydd Road, CF24
4AG Cardiff, Wales, UK.} \quad Partially supported by EPSRC grant EP/D054621.}

\thanks{{\it E-mail addresses:} \texttt{delta4@math.spbu.ru, KorotyaevE@cf.ac.uk}}

\begin{abstract}
The matrix-valued Weyl-Titchmarsh functions $M(\l)$ of vector-valued Sturm-Liouville operators
on the unit interval with the Dirichlet boundary conditions are considered. The collection of
the eigenvalues (i.e., poles of $M(\l)$) and the residues of $M(\l)$ is called the spectral
data of the operator. The complete characterization of spectral data (or, equivalently, $N\ts
N$ Weyl-Titchmarsh functions) corresponding to $N\ts N$ self-adjoint square-integrable
matrix-valued potentials is given, if all $N$ eigenvalues of the averaged potential are
distinct.
\end{abstract}

\maketitle

\section{Introduction}

We start with a short description of known results in the inverse spectral theory for
{\it scalar} Strum-Liouville operators on a {\it finite} interval. We recall only some
important steps mostly focusing on the {\it characterization} problem, i.e., the complete
description of spectral data that correspond to some fixed class of potentials. More
information about different approaches to inverse spectral problems can be found in the
monographs \cite{MaBook}, \cite{LeBook}, \cite{PT}, \cite{FY}, survey \cite{Ges} and
references therein.

The inverse spectral theory goes back to the seminal \mbox{paper \cite{Bo}} (see also
\cite{Levinson}). Borg showed that spectra of two Sturm-Liouville problems $-y''+q(x)y=\l y$,
\mbox{$x\in [0,1]$}, with the same boundary conditions at $1$ but different boundary
conditions at $0$, determine the potential $q(x)$ and the boundary conditions uniquely.
Later on, Marchenko \cite{Ma1} proved that the so-called spectral function $\r(\l)$ (or,
equivalently, the Weyl-Titchmarsh function $m(\l)$) determines the potential uniquely. Note
that the spectral function is piecewise-linear outside the spectrum $\{\l_n\}_{n=1}^{+\iy}$
and its jump at $\l_n$ is equal to the so-called {\it normalizing constant} $[\a_n(q)]^{-1}$
given by (\ref{AnScalar}). At the same time, a different approach to this problem was
developed by \mbox{Krein \cite{Kr1}, \cite{Kr1a}, \cite{Kr2}}.

An important result was obtained by Gel'fand and Levitan \cite{GL}. They gave an effective
method to reconstruct the potential $q$ from its spectral function. More precisely, they
derived an integral equation and expressed $q(x)$ explicitly in terms of the solution of
this equation. At that time, there was some gap between necessary and sufficient conditions
for the spectral functions corresponding to fixed classes of $q(x)$.

Some characterization of spectral data for $q$ such that $q^{(m)}\in \cL^1(0,1)$ was derived
by Levitan and Gasymov \cite{LG} for all $m=0,1,2,..$. Also, they gave the solution of the
characterization problem in the case $q''\in\cL^2(0,1)$. Marchenko and Ostrovski \cite{MO}
obtained a sharpening of this result. Namely, for all $m=0,1,2,..$ they gave the complete
solution of the inverse problem in terms of two spectra, if $q^{(m)}\in \cL^2(0,1)$.

Trubowitz and co-authors (Isaacson \cite{IT}, McKean \cite{IMT}, Dahlberg \cite{DT}, P\"oschel
\cite{PT}) suggested another approach. 
It is based on the analytic properties of the mapping
$\mathrm{\{potentials\}\mapsto\{spectral\ data\}}$ and the explicit transforms corresponding
to the change of only  a {\it finite} number of spectral parameters
\mbox{$(\l_n(q),\n_n(q))_{n=1}^{+\iy}$}. Their norming constants $\n_n(q)$  differ slightly
from the normalizing constants (\ref{AnScalar}), but the characterizations are equivalent (see
Appendix \ref{AppendixB}). Also, this approach was applied to other scalar inverse problems
with purely discrete spectrum (singular Sturm-Liouville operator on $[0,1]$ \cite{GR};
perturbed harmonic oscillator \cite{MT}, \cite{CKK}, \cite{CK1}).

Thus, nowadays the inverse spectral theory for the {\it scalar} Sturm-Liouville operators is
well understood. By contrast, until recently only some particular results were known for
 {\it vector-valued operators}.

\smallskip

In our paper we consider the inverse problem for the self-adjoint operators
\begin{equation}
\label{Hdef} \bL\p=-\p''+V(x)\p,\qquad \p(0)=\p(1)=0,\quad \p\in\cL^2([0,1];\C^N),
\end{equation}
where $V=\!V^*\!\in\!\cL^2([0,1];\C^{N\ts N})$ is a self-adjoint $N\ts N$ matrix-valued
potential. Denote by $\vp(x)=\vp(x,\l,V)$ and $\c(x)=\c(x,\l,V)$ the matrix-valued solutions
of the equation $-\p''+V(x)\p=\l\p$ such that
\[
\vp(0)=\c(1)=0, \qquad  \vp'(0)=-\c'(1)=I_N,
\]
here and below $I_N$ denotes the identity $N\!\ts\!N$ matrix. Note that
\[
\c(x,\l,V)=\vp(1\!-\!x,\l,V^\sharp),\qquad \mathrm{where}\qquad V^\sharp(x)\equiv V(1\!-\!x),\
x\in [0,1].
\]
The {\it matrix-valued Weyl-Titchmarsh function} for this problem is given by
\begin{equation}
\label{Mdef} M(\l)=M(\l,V)=[\c'\c^{-1}](0,\l,V)=[M(\ol\l)]^*,\quad \l\in\C.
\end{equation}
In the scalar case, the Weyl-Titchmarsh function $m(\l,q)$ is a meromorphic function having
simple poles at Dirichlet eigenvalues $\l_n(q)$ and
\begin{equation}
\label{AnScalar} \res_{\l=\l_n(q)}m(\l,q)=-[\a_n(q)]^{-1}=-\lt[\int_0^1
|\vp(x,\l_n,q)|^2dx\rt]^{-1}.
\end{equation}
So, the sharp characterization of all {\it scalar} Weyl-Titchmarsh functions (or,
equivalently, all spectral data $(\l_n(q),\a_n(q))_{n=1}^{+\iy}$\,) that correspond to
potentials $q\in\cL^2(0,1)$ is available due to \cite{MO} or \cite{PT} (see also
Appendix~\ref{AppendixB}). Namely, the necessary and sufficient conditions are
\begin{equation}
\label{CondScalar}
\begin{array}{cl}
\l_1<\l_2<\l_3<...,\qquad & (\l_n-\pi^2n^2-q_0)_{n=1}^{+\iy}\in\ell^2\quad \mathrm{for\ some}\
\ q_0\in\R \cr \mathrm{and}\vphantom{|^\big|} &(\pi n\cdot
(2\pi^2n^2\a_n(q)-1))_{n=1}^{+\iy}\in\ell^2.
\end{array}
\end{equation}

In the vector-valued case, it is known that the Weyl-Titchmarsh function determines $V$
uniquely (see \cite{Mal} or \cite{Yu}). Some other miscellaneous results concerning
vector-valued Schr\"odinger operators were obtained in \cite{Ca}, \cite{CKper}, \cite{ChSh},
\cite{CHGL}, \cite{JL1}, \cite{JL2}, \cite{SP}, \cite{Sh}. Nevertheless, to the best of our
knowledge, no solutions of the {\it characterization} problems have been available until recently.

Following \cite{CK}, we denote by $\l_1<\l_2<..<\l_\a<...$ the {\it eigenvalues} of $\bL$ and
by $k_\a=\dim\bE_\a\in [1,N]$ their {\it multiplicities}, where $\bE_\a\ss\cL^2([0,1];\C^N)$
is the eigenspace corresponding to the eigenvalue $\l_\a$. Then (see details in \cite{CK}),
the Weyl-Titchmarsh function $M(\l)$ is meromorphic outside the Dirichlet spectrum
$\s(V)=\{\l_\a(V)\}_{\a\ge 1}$ and
\[
\res_{\l=\l_\a}M(\l)= -B_\a= - p_\a^* g_\a^{-1}p_\a,
\]
where
\[
p_\a:\C^N\to\cE_\a =
\Ker\vp(1,\l_\a,V)=\left\{h\in\C^N:\p_{\a;h}=\vp(\cdot,\l_\a,V)h\in\bE_\a\right\}
\]
is the orthogonal {\it projector} and
\[
g_\a= p_\a \lt[\int_0^1[\vp^*\vp](x,\l_\a,V)dx\rt] p_\a^* = g_\a^*>0
\]
is the self-adjoint operator (or the {\it normalizing matrix}) acting in $\cE_\a$. We also use
the notation $P_\a=p_\a^*p_\a^{\,}:\C^N\to\cE_\a\ss\C^N$. Note that for all $h_1,h_2\in\cE_\a$
one has
\[
\langle\p_{\a;h_1},\p_{\a;h_2}\rangle_{\cL^2([0,1];\C^N)} = \int_0^1
h_2^*[\vp^*\vp](x,\l_\a,V)h_1\,dx=\langle h_1,g_\a h_2\rangle_{\cE_\a}\,.
\]
We call $(\l_\a,P_\a,g_\a)_{\a=1}^{+\iy}$ the {\it spectral data} of the operator $\bL$. If
$k_\a=1$, then $g_\a$ acts in the one-dimensional space $\cE_\a$, so we consider it as a
positive real number (and call it, as in the scalar case, the {\it normalizing constant}). The
spectral data determine (e.g., see Proposition \ref{Mformula}) the function $M(\l)$, and so
the potential $V(x)$, uniquely. The main result of our paper is the following solution of the
{\it characterization} problem.

Let $e_1^0,e_2^0,..,e_N^0$ be the standard coordinate basis and $P_j^0=\langle\cdot,
e_j^0\rangle e_j^0$ be the coordinate projectors in $\C^N$. We denote the Euclidian norm of
vectors $h\in\C^N$ and the operator norm of matrices $A\in\C^{N\ts N}$ by $|h|$ and $|A|$,
respectively.

\begin{theorem}[\bf Characterization of spectral data]
\label{MainThm} For all $v_1^0<v_2^0<..<v_n^0$ the mapping
$V\mapsto(\l_\a,P_\a,g_\a)_{\a=1}^{+\iy}$ is a bijection between the space of potentials
\begin{equation}
\label{Vnondeg} V\!=\!V^*\in \cL^2([0,1];\C^{N\ts N})\quad\mathit{such\ that}\ \int_0^1V(x)dx=
\diag\{v_1^0,v_2^0,..,v_N^0\}
\end{equation}
and the class of spectral data satisfying the following conditions (A)-(C):

\smallskip

\noindent {\rm (A)} The spectrum is asymptotically simple, i.e., there exist $\a^\dia\ge 0$,
$n^\dia\ge 1$ such that
\[
k_1^\dia+k_2^\dia+..+k_{\a^\dia}^\dia=N(n^\dia\!-\!1)\quad \mathit{and} \quad k_\a^\dia=1\ \
\mathit{for\ all}\ \a\ge\a^\dia\!+\!1.
\]
It allows us to define the double-indexing $(n,j)$, $n\!\ge\!n^\dia$, $j\!=\!1,2,..,N$,
instead of
$\a\!>\!\a^\dia$. 
Namely, we set $\l_{n,j}=\l_{\a^\dia+N(n-n^\dia)+j}$, $P_{n,j}=P_{\a^\dia+N(n-n^\dia)+j}$ and
so on for $n\!\ge\!n^\dia$.

\smallskip

\noindent {\rm (B)} The following hold true for all $j=1,2,..,N$:
\begin{equation}
\label{AsymptInThm}
\begin{array}{lll}
(\lambda_{n,j}\!-\!\pi^2n^2\!-\!v_j^0)_{n=n^\dia}^{+\infty}\in\ell^2,&& (\pi
n\cdot(2\pi^2n^2g_{n,j}\!-\!1))_{n=n^\dia}^{+\infty}\in\ell^2, \cr
(|P_{n,j}\!-\!P_j^0|)_{n=n^\dia}^{+\infty}\in\ell^2\vphantom{|^\big|}&\mathit{and}\ \ & (\pi
n\cdot|{\textstyle\sum_{j=1}^N} P_{n,j}\!-\!I_N|)_{n=n^\dia}^{+\infty}\in\ell^2.
\end{array}
\end{equation}

\smallskip

\noindent {\rm (C)} The collection $(\lambda_\a\,;P_\a)_{\a=1}^{+\infty}$ satisfies the
following property:

\smallskip

\begin{quotation}
\noindent Let $\xi:\C\to \C^N$ be an entire vector-valued function. If
$P_\a\xi(\lambda_\a)\!=\!0$ for all $\a\!\ge\! 1$,\ \
$\xi(\lambda)\!=\!O(e^{|\Im\sqrt\lambda|})$ as $|\lambda|\to\infty$\ \ and\ \ $\xi\in
\cL^2(\R_+)$,\ \ then $\xi(\l)\equiv 0$.
\end{quotation}
\end{theorem}

\begin{remark}
\label{Degeneracy} {Let $V=V^*\in\cL^2([0,1];\C^{N\ts N})$. Applying some unitary transform in
$\C^N$, one may always assume that $\int_0^1V(x)dx=\diag\{v_1^0,v_2^0,..,v_N^0\}$, $v_1^0\le
v_2^0\le..\le v_N^0$. Our assumption (\ref{Vnondeg}) states that all the $v_j^0$ are distinct. It
simplifies the analysis, since otherwise infinitely many eigenvalues $\l_\a$ can be multiple.
In particular, in the general case, one has to introduce some other parameters instead of
$(P_{n,j},\,g_{n,j})$.}
\end{remark}

We give also a simple reformulation of the algebraic restriction (C) (note that it doesn't
depend on the shift of the spectrum).

\begin{proposition}[\bf reformulation of (C)]
\label{Creform} Let $\l_\a>0$ for all $\a\ge 1$ and $P_\a = h_\a h_\a^*$, where
$h_\a=(h_\a^{(1)};..\,;h_\a^{(k_\a)})$ consists of $k_\a$ orthonormal vectors
$h_\a^{(j)}\in\C^N$. Then the condition (C) is equivalent to the following:
\begin{quotation}
\noindent Vector-valued functions $e^{\pm i\sqrt{\l_\a}t}h_\a^{(j)}$, $j=1,..,k_\a$, $\a\ge
1$, together with the constant vectors $e_1^0,..,e_N^0$ span $\cL^2([-1,1]\,;\C^N)$.
\end{quotation}
\end{proposition}
\begin{remark}
In the scalar case, (C) always holds true due to the well known result of Paley and Wiener
(e.g., see \cite{Le} p.47). In the vector-valued case, this condition is not trivial. Some
discussion of (C) is given in Appendix \ref{AppendixA} (see Propositions \ref{xibDef},
\ref{FaAnd(C)}). Note that, if $P_{n,j}=P_j^0$ for all $n\ge m +1$ and $j=1,2,..,N$, then one
can reformulate (C) as the condition $\det \cT\ne 0$ for some $Nm\!\ts\!Nm$ matrix $\cT$ (see
Proposition \ref{FiniteProj}).
\end{remark}

As usual, Theorem \ref{MainThm} consists of several different parts:

\noindent (i) Uniqueness Theorem (spectral data determine the potential uniquely);

\noindent (ii) Direct Problem (spectral data constructed by a given potential satisfy
(A)-(C));

\noindent (iii) Surjection (any data satisfying (A)-(C) are spectral data of some potential).

We do not discuss the uniqueness theorem (i) in our paper and refer to \cite{Mal}, \cite{Yu}
(or \cite{CK}) for this fact. The direct problem (ii) is considered in
Sect.~\ref{SectDirectProblem}. Note that the spectrum is asymptotically simple due to our
assumption $v_1^0<v_2^0<..<v_N^0$ (see also Remark \ref{Degeneracy}).  As in the scalar case,
the Fourier coefficients of $V$ appear as leading terms in the asymptotics of the spectral
data (Propositions \ref{RoughAsymptEV} and \ref{AsymptBn}). We also give the explicit
expression for
$M(\l)$ in terms of the spectral data in Sect. \ref{SectMFormula}.

The main part of our paper (Sect. \ref{SectInverseProblem}) is devoted to the surjection
(iii). The general strategy of the proof is described in detail in
Sect.~\ref{SectGeneralStrategy}. Here we give only a short sketch of our arguments. We start
with some admissible data $(\l_\a^\dia,P_\a^\dag,g_\a^\dag)_{\a\ge 1}$ satisfying
\mbox{(A)--(C)}. Using the well known characterization (\ref{CondScalar}) for the scalar case,
we construct some special {\it diagonal} potential $V^\dia$ such that
$\s(V^\dia)=\{\l_\a^\dia\}_{\a\ge 1}$.

In \mbox{Sect. \ref{aBnjRoughAsympt}--\ref{SectAnalyticity2}} we introduce some essential {\it
modification} of the spectral data in order (a) to control the splitting of multiple
eigenvalues and (b) to join together all asymptotics in (\ref{AsymptInThm}). We prove that the
mapping $\F:\mathrm{\{potentials\}\mapsto\{modified\ spectral\ data\}}$ is
real-analytic\begin{footnote}{ The mapping $F:U\to H^{(2)}$ between {\it real} Hilbert spaces
$U\ss H^{(1)}$ and $H^{(2)}$ is {\it real-analytic} iff it has continuation $F_\C:U_\C\to
H^{(2)}_\C$ into some {\it complex} neighborhood $U\!\ss\!U_\C\!\ss\!H^{(1)}_\C$ that is {\it
differentiable} as the mapping between the {\it complexifications} $H^{(1)}_\C$, $H^{(2)}_\C$
of the real spaces $H^{(1)}$, $H^{(2)}$.}
\end{footnote}
near $V^\dia$. The main purpose of involving analyticity arguments here is the well known
equivalence of the analyticity and the weak-analyticity\begin{footnote}{In Hilbert spaces, the
weak-analyticity is equivalent to the analyticity of particular coordinates and the local
boundedness, see nice Appendix A in \cite{PT} or the monograph \cite{Di} for
details.}\end{footnote} for mappings between {\it complex} Hilbert spaces. Thus, we
immediately derive the smoothness of the whole mapping $\F$ from the smoothness of its
components.

In \mbox{Sect. \ref{SectFrechetDer1}, \ref{SectFrechetDer2}} we use the Fredholm Alternative
in order to show that $\F$ is a {\it local isomorphism} near $V^\dia$ (i.e., $d_{V^\dia}\F$ is
invertible). Thus, all additional spectral data sufficiently close to
$(P_\a(V^\dia),g_\a(V^\dia))_{\a\ge 1}$ can be obtained from potentials having the same
spectrum $\{\l_\a^\dia\}_{\a\ge 1}$ as $V^\dia$. In particular, if $\a^\bullet$ is large
enough, then there exists $V^\bullet$ such that $\s(V^\bullet)=\{\l_\a^\dia\}_{\a\ge 1}$ and
$(P_\a(V^\bullet),g_\a(V^\bullet))=(P_\a^\dag,g_\a^\dag)$ for all $\a>\a^\bullet$.

We complete the proof in Sect.~\ref{SectChangingFirst} using the explicit isospectral
transforms constructed in our recent paper \cite{CK}. As usual in Trubowitz's approach, we
need to change only some {\it finite} number $\a^\bullet$ of additional spectral data
$(P_\a,g_\a)$. Note that the condition (C) and the restrictions introduced in \cite{CK} in
terms of "forbidden" subspaces are equivalent (see Proposition \ref{FaAnd(C)}). Thus, one can
change any finite number of projectors $P_\a$ in an arbitrary way that doesn't violate (C)
(see details in Sect.~\ref{SectChangingFirst}).

Note that we do not present any explicit {\it reconstruction procedure} for the potential, if
there are {\it infinitely} many perturbed spectral data. The natural idea is to use some
passage to the limit changing the residues $B_\a(V^\dia)\mapsto B_\a^\dag$, $\a=1,2,..$, of
the Weyl-Titchmarsh function step by step. Each step is doable due to isospectral transforms
constructed in \cite{CK} but we do not prove the {\it convergence} of this procedure.

\smallskip

We finish the introduction with several remarks concerning some possible further developments of
our approach to this inverse problem.

\begin{remark}
{\rm The isospectral transforms constructed in \cite{CK} generalize the scalar isospectral
flows (see~\cite{PT}) and some specific class of isospectral transforms given in \cite{JL1}.
Nevertheless, to the best of our knowledge, no analogues of the explicit flows changing the
eigenvalues (see~\cite{PT}) are known in the vector-valued case. We think that such a
construction would simplify the inverse theory a lot. }
\end{remark}

\begin{remark}
{\rm One may be interested in the characterization for other parameters, e.g. the spectra of
several boundary problems (similarly to the original paper \cite{Bo}). \mbox{Almost} nothing
is known here. Yurko \cite{Yu} proved that $N^2\!+\!1$ spectra determine the potential
uniquely. On the other hand, the naive count says that this inverse problem is overdetermined.
Note that, in the spirit of Appendix \ref{AppendixB}, this question can be considered as a
parametrization problem for some class of matrix-valued functions.}
\end{remark}

\begin{remark}
{\rm Consider the Schr\"odinger operator $Hy=-y''+Vy$ on $\R$ with a
$N\!\times\!N$ potential $V=V^*$ such that $\int_\R (1+|x|)|V(x)|dx<+\iy$  (e.g., see \cite{O}). It has a finite
number of eigenvalues $\l_1<..<\l_m<0$ with the multiplicities $k_\a=\dim \bE_\a$, where
$\bE_\a$ is the eigenspace corresponding to $\l_\a$. In order to solve the inverse scattering
problem completely, one needs to characterize the residues of the transmission coefficient
\mbox{at $\l_\a$}. Unfortunately, we do not know any results in this direction. For the
scattering problem on the half-line a characterization was given in \cite{AM} but it
involves {\it implicit} conditions for spectral data (much more complicated than our condition
(C)). }
\end{remark}

\begin{remark}
{\rm In the scalar case, the Dirichlet eigenvalues and the norming constants are canonically
conjugate variables for the Korteweg-de Vries equation with periodic initial conditions (see
\cite{FM}). Similarly, the (negative) eigenvalues and the corresponding normalizing constants
of the (scalar) Schr\"odinger operator $-y''+q(x)y$ on $\R$ with a decreasing potential $q(x)$
are canonically conjugate variables for the Korteweg-de Vries equation (see \cite{ZF}). The
vector-valued case is  more complicated (see \cite{CD1}, \cite{CD2}, \cite{O}). We hope that
our results could be useful from this point of view. }
\end{remark}

\smallskip

\noindent {\bf Acknowledgements.} Some parts of this paper were written at Mathematisches
Forschungsinstitut Oberwolfach, Institut fu\"r Mathematik Humboldt-Universit\"at zu Berlin and
Section de Math\'ematiques Universit\'e de Gen\`eve. The authors are grateful to the
Institutes for the hospitality. The stay of the authors at the MFO was provided by the
Oberwolfach-Leibniz Fellowship of the first author. The authors are grateful to the MFO for
its stimulating atmosphere.

\section{Direct problem}
\setcounter{equation}{0} \label{SectDirectProblem}

\subsection{Asymptotics of the eigenvalues and the individual projectors.}
\label{SectAsymptRough} Denote by
\[
\nh V 0 = \int_0^1 V(t)dt,\quad \nh V {cn} = \int_0^1V(t)\cos 2\pi n t\,dt\quad
\mathrm{and}\quad \nh V{sn} = \int_0^1 V(t)\sin 2\pi nt\,dt
\]
the (matrix) Fourier coefficients of $V$. We start with some elementary asymptotics of the
fundamental solutions $\vp(x,\l,V)$ and \mbox{$\c(x,\l,V)=\vp(1-x,\l,V^\sharp)$} for $\l$
close to~$\pi^2n^2$. It's well known that
\begin{equation}
\label{xVpasympt} \vp(x,z^2,V)=\frac{\sin zx}{z}\,I_N+ \frac{1}{z^2}\int_0^x\sin
z(x\!-\!t)\cdot V(t)\sin zt\, dt + O\lt(\frac{e^{|\Im z|x}}{|z|^3}\rt).
\end{equation}
Here and below constants in $O$--type estimates depend on the potential. In this section we do
not pay the attention to the nature of this dependence. Let
\[
z^2=\pi^2n^2\!+\m,\ \ \m=O(1),\qquad \mathrm{so}\qquad z= \pi n + \frac{\m}{2\pi n} +
O\lt(\frac{1}{n^3}\rt).
\]
Then,
\[
\vp(x,z^2,V)= \frac{\sin zx}{\pi n}\,I_N + \frac{1}{\pi^2n^2}\int_0^x \sin\pi
n(x\!-\!t)\cdot V(t)\sin \pi nt\, dt+O\lt(\frac{1}{n^3}\rt).
\]
In particular,
\begin{equation}
\label{Vp1Asympt} \vp(1,z^2,V)=\frac{(-1)^{n}}{2\pi^2n^2}\left[\m I_N - \nh V0 + \nh
V{cn}+O\lt(\frac{1}{n}\rt)\right].
\end{equation}

\begin{proposition}
\label{RoughAsymptEV} Let $V=V^*\in \cL^2([0,1];\C^{N\ts N})$ satisfy $\nh
V0=\diag\{v_1^0,v_2^0,..,v_N^0\}$ with $v_1^0<v_2^0<..<v_N^0$. Then,

\smallskip

\noindent (i) there exists $n^\dia=n^\dia(V)\ge\|V\|$ such that (a) there are exactly
$N(n^\dia\!-\!1)$ eigenvalues counting with multiplicities in the interval
$(-\pi^2(n^\dia\!-\!1)^2-3\|V\|; \pi^2(n^\dia\!-\!1)^2+3\|V\|)$, (b) \nolinebreak for each
$n\ge n^\dia$ there are exactly $N$ simple eigenvalues $\l_{n,1},\l_{n,2},..,\l_{n,N}$ in the
interval \mbox{$(\pi^2n^2-3\|V\|;\pi^2n^2+3\|V\|)$,} (c) \nolinebreak there are no other
eigenvalues;

\smallskip

\noindent (ii) for each $j=1,2,..,N$ the following asymptotics hold true as $n\to\iy$:
\[
\l_{n,j}=\pi^2n^2+v_j^0- \nh v{cn}_{jj} + O(\d_n(V)),\quad\mathit{where}\quad \d_n(V)=|\nh
V{cn}|^2+\frac{1}{n};
\]
(iii) if $p_{n,j}=\langle\cdot,h_{n,j}\rangle h_{n,j}$, where $h_{n,j}\in\C^N$ is such that
$|h_{n,j}|=1$, $\langle h_{n,j},e_j^0 \rangle>0$, then the asymptotics
\[
h_{n,j}= \left(\begin{array}{ccccccc}\displaystyle \frac{\nh v{cn}_{1,j}}{v_1^0-v_j^0} &
\displaystyle .. & \displaystyle \frac{\nh v{cn}_{j-1,j}}{v_{j-1}^0-v_j^0} & \displaystyle 1 &
\displaystyle \frac{\nh v{cn}_{j+1,j}}{v_{j+1}^0-v_j^0} & \displaystyle .. & \displaystyle
\frac{\nh v{cn}_{N,j}}{v_N^0-v_j^0}\end{array}\right)^\top + O(\d_n(V))
\]
hold true for each $j=1,2,..,N$ as $n\to\iy$.
\end{proposition}

\noindent Note that the condition $n^\dia(V)\ge\|V\|$ guarantees that the mentioned intervals
do not intersect each other. We need the following simple matrix version of Rouche's Theorem:

\begin{lemma}
\label{MatrixRouche} Let $F,G:\ol{B(w,r)}\to\C$ be analytic matrix-valued functions such that
$|G(\l)|\cdot|F^{-1}(\l)|<1$ for all $\l$ on the boundary of some disc $\ol{B(w,r)}\ss\C$.
Then, the scalar functions $\det F$ and $\det(F\!+\!G)$ have the same number of zeros in
$B(w,r)$ counting with multiplicities.
\end{lemma}
\begin{proof}
We check that $\Delta_\cC\arg (\det F)=\Delta_\cC\arg (\det (F\!+\!G))$, where $\Delta_\cC
\arg f$ denotes the increment of $\arg f$ along the circumference $\cC=\{\l:|\l\!-\!w|=r\}$.
Note that, if $\l\in \cC$, then all eigenvalues of $I\!+\!G(\l)F^{-1}(\l)$ have strictly
positive real parts since $|G(\l)F^{-1}(\l)|<1$. Thus, the result follows from
\[
\Delta_\cC\arg (\det (F\!+\!G))-\Delta_\cC\arg (\det F)= \Delta_\cC\arg (\det
(I\!+\!GF^{-1}))=0
\]
and the classical argument principle.
\end{proof}

\begin{proof}[{Proof of Proposition \ref{RoughAsymptEV}}]
(i) Firstly, we apply Lemma \ref{MatrixRouche} to the function
\[
\c(0,\l,V)=\vp(1,\l,V^\sharp)=F(\l)+G(\l)
\]
in the discs
\[
\{\l:|\l|<\pi^2 n^2\!+\!3\|V\|\} \quad \mathrm{with} \quad
F(\l)=\frac{\sin\sqrt{\l}}{\sqrt\l}\,I_N
\]
(see asymptotics (\ref{xVpasympt})) and
\[
\{\l:\l=\pi^2n^2+\m,\ |\m|<3\|V\|\} \quad \mathrm{with} \quad \displaystyle F(\l)=
\frac{(-1)^{n}}{2\pi^2n^2}\left((\l\!-\!\pi^2n^2)I-\nh V0\right)
\]
(see asymptotics (\ref{Vp1Asympt})). Thus, if $n$ is sufficiently large, then there are
exactly $Nn$ and $N$ eigenvalues (zeros of $\det\c(0,\cdot,V)$), respectively, inside these
discs counting with multiplicities. Secondly, let
\[
\textstyle d=\frac{1}{2}\min_{j=1,..,N-1} (v_{j+1}^0-v_j^0).
\]
If $n$ is sufficiently large, then $|\nh V{cn}|$ is small and one can apply Lemma
\ref{MatrixRouche} (with the same functions $F$ as above) in the discs
\[
\{\l:\l=\pi^2n^2+v_j^0+\m,\ |\m|<d\},\qquad j=1,2,..,N.
\]
So, if $n\ge n^\dia$, then there are exactly one simple eigenvalue
$\l_{n,j}=\pi^2n^2\!+\!\m_{n,j}$ inside each small disc $B(\pi^2n^2\!+\!v_j^0,d)$ and there
are no other eigenvalues.

\smallskip

\noindent (ii) Recall that $\det\vp(1,\l_{n,j},V)=0$. Therefore, due to (\ref{Vp1Asympt}) and
the standard perturbation theory, the self-adjoint matrix $\m_{n,j}I_N-\nh V0 + \nh V{cn}$ has
at least one eigenvalue $\t$ such that $|\t|=O(n^{-1})$. On the other hand, the eigenvalues of
the matrix $\nh V0 - \nh V{cn}$ are $\t_s=v_s^0-\nh v{cn}_{ss} + O(|\nh V{cn}|^2)$,
$s=1,2,..,N$. Hence, for some $s$,
\[ \m_{n,j}-v_s^0+\nh
v{cn}_{ss} = O(|\nh V{cn}|^2) + O(n^{-1}).
\]
Due to (i), $s=j$.

\smallskip

\noindent (iii) Let $j=1$ for the simplicity and $d_k^0=v_1^0-v_k^0$, $k\!=\!2,..,N$. In view
of (\ref{Vp1Asympt}) and (ii),
\[
\vp(1,\l_{n,1},V)= \frac{(-1)^{n}}{2\pi^2n^2}\left(\begin{array}{cccc} 0 & \nh v{cn}_{12} &
... & \nh v{cn}_{1N} \cr \nh v{cn}_{21} & d_2^0\!-\nh v{cn}_{11}\!\!+\!\nh v{cn}_{22} & ... &
\nh v{cn}_{2N} \cr ... & ... & ... & ... \cr \nh v{cn}_{N1} & \nh v{cn}_{2N} & ... &
d_N^0\!-\nh v{cn}_{11}\!\!+\!\nh v{cn}_{NN}
\end{array}\right)\! + O\lt(\frac{\d_n(V)}{n^2}\rt).
\]

\noindent Recall that $\vp(1,\l_{n,1},V)h_{n,1}=0$. Thus,

\smallskip

$\langle\vp(1,\l_{n,1},V)h_{n,1},e_k^0\rangle=0$ gives $\langle h_{n,1},e_k^0\rangle = O(|\nh
V{cn}|+\d_n(V))$ for all $k=2,..,N$,

\smallskip

$|h_{n,1}|=1$ gives $\langle h_{n,1},e_1^0\rangle=1+O(\d_n(V))$

\smallskip

\noindent and, using $\langle\vp(1,\l_{n,1},V)h_{n,1},e_k^0\rangle=0$ again, one obtains
\[
\nh v{cn}_{k1} + d_k^0\cdot \langle h_{n,1},e_k^0\rangle + O(\d_n(V))=0,\quad k=2,..,N.
\]
Note that (ii), (iii) are standard results for the perturbation of a simple eigenvalue.
\end{proof}

\subsection{Asymptotics of the norming constants and the averaged projectors}
\label{SectAsymptFine} 
Due to Proposition \ref{RoughAsymptEV}, all sufficiently large eigenvalues are simple.
Therefore, for all sufficiently large $n\ge n^\dia$ and $j=1,2,..,N$ we may introduce the
factorization
\[
P_{n,j}=h_{n,j}h_{n,j}^*,\qquad B_{n,j}=-\res_{\l=\l_{n,j}}M(\l) =
h_{n,j}g_{n,j}^{-1}h_{n,j}^*=g_{n,j}^{-1}P_{n,j},
\]
where $g_{n,j}>0$, $h_{n,j}\in\C^N$, $|h_{n,j}|=1$ and $\langle h_{n,j},e_j^0\rangle
>0$. Denote
\[
B_n=B_n(V)=\sum_{j=1}^N B_{n,j},\quad n\ge n^\dia.
\]

We begin with some simple reformulations of the needed asymptotics. Note that Proposition
\ref{RoughAsymptEV} gives
\begin{equation}
\label{ProjRoughAsymp} h_{n,j}=e_j^0+\ell^2\quad \mathrm{for\ all}\quad j=1,2,..,N.
\end{equation}
Here and below we write $a_n=b_n+\ell^2_k$ iff
\[
(|a_n\!-\!b_n|)_{n=n^\dia}^{+\iy}\in \ell^2_k=\left\{(c_n)_{n=n^\dia}^{+\iy}:
(n^kc_n)_{n=n^\dia}^{+\iy}\in\ell^2\right\}.
\]
\begin{lemma}
\label{EquivAsymptP}
The following asymptotics are equivalent:
\[
\begin{array}{ll}
(i)  & \sum_{j=1}^N P_{n,j}=I_N+\ell^2_1;\vphantom{|_\big|}\cr (ii) & \langle
h_{n,j},h_{n,k}\rangle=\ell^2_1\quad for\ all\ \ j\ne k,\ j,k=1,2,..,N.
\end{array}
\]
\end{lemma}

\begin{proof}
Introduce $N\ts N$ matrices $h_n=(\begin{array}{ccccccc}h_{n,1} &;& h_{n,2} &;& ... &;&
h_{n,N}\end{array})$. Then
\[
h_n h_n^* = \sum_{j=1}^N h_{n,j} h_{n,j}^* = \sum_{j=1}^N P_{n,j}\\[-4pt]
\]
and
\[
h_n^*h_n = \left(h_{n,j}^*\,h_{n,k}\right)_{j,k=1}^N = \left(\langle
h_{n,k}\vphantom{h_{n,j}^*}, h_{n,j} \rangle\right)_{j,k=1}^N\,.
\]
The matrices $h_n h_n^*$ and $h_n^* h_n$ are unitary equivalent (since
$h_nh_n^*=u_n(h_n^*h_n)u_n^*$, where $h_n=u_ns_n$ is the polar decomposition of $h_n$). Thus,
the asymptotics $h_n h_n^*= I_N+\ell^2_1$ are equivalent to the asymptotics
$h_n^*h_n=I_N+\ell^2_1$ (note that $\langle h_{n,j}, h_{n,j}\rangle =|h_{n,j}|^2=1$).
\end{proof}

\begin{lemma}
\label{EquivAsymptB} The collection of asymptotics
\[
g_{n,j}^{-1}=2\pi^2n^2(1+\ell^2_1)\quad for\ all\ j=1,2,..,N\quad and \quad \sum_{j=1}^N
P_{n,j}=I_N+\ell^2_1
\]
is equivalent to
\[
B_n=2\pi^2 n^2 (I+\ell^2_1).
\]
\end{lemma}
\begin{proof}
As in Lemma \ref{EquivAsymptP}, we set $H_n=(\begin{array}{ccccccc}
g_{n,1}^{-\frac{1}{2}}h_{n,1} &;& g_{n,2}^{-\frac{1}{2}}h_{n,2} &;& ... &;&
g_{n,N}^{-\frac{1}{2}}h_{n,N}\end{array})$. Note that $B_n= H_n H_n^*$ while
\[
H_n^*H_n = \left(g_{n,j}^{-\frac{1}{2}}g_{n,k}^{-\frac{1}{2}}\cdot \langle h_{n,k}, h_{n,j}
\rangle\right)_{j,k=1}^N\,.
\]
Thus, as above, asymptotics $B_n=2\pi^2n^2(I_N+\ell^2_1)$ and
$H_n^*H_n=2\pi^2n^2(I_N+\ell^2_1)$ are equivalent. The diagonal entries of $H_n^*H_n$ are
$g_{n,j}^{-1}$, so $g_{n,j}^{-1}=2\pi^2n^2(1+\ell^2_1)$. Asymptotics of the non-diagonal
entries give $\langle h_{n,k}, h_{n,j} \rangle=
2\pi^2n^2g_{n,j}^{1/2}g_{n,k}^{1/2}\cdot\ell^2_1=\ell^2_1$, $j\ne k$, which is equivalent to
$\sum_{j=1}^N P_{n,j}=I_N+\ell^2_1$ due to Lemma \ref{EquivAsymptP}.
\end{proof}

Note that, for sufficiently large $n$,
\[
B_n(V)=
-\sum_{j=1}^N\res_{\l=\l_{n,j}}M(\l)= -\frac{1}{2\pi i}\oint_{|\l-\pi^2n^2|=3\|V\|}M(\l)d\l.
\]
This formula allows us to determine sharp asymptotics of $B_n(V)$. Moreover, it defines the
analytic continuation of $B_n(V)$ for non-selfadjoint potentials.

\begin{proposition}
\label{AsymptBn} The following asymptotics hold true
\[
B_n(V)=2\pi^2n^2\lt[I_N-\frac{1}{\pi n}\,\nh{[(1\!-\!t)V]}{sn} + O\lt(\frac{1}{n^2}\rt)\rt]
\]
uniformly on bounded subsets of potentials  $V\in \cL^2([0,1];\C^{N\ts N})$.
\end{proposition}

\begin{proof}
It's well known that
\[
\begin{array}{l}
\displaystyle \c(0,z^2,V) = \vp(1,z^2,V^\sharp) = \frac{\sin z}{z}\,I_N +
\frac{1}{z^2}\int_0^1\!\sin z(1\!-\!t)\cdot V^\sharp(t)\sin zt\, dt
\vphantom{\int_\big|}\cr\displaystyle\hphantom{\c(0,z^2,V)} + \frac{1}{z^3}\int_0^1\!dx\sin
z(1\!-\!x)\cdot V^\sharp(x)\int_0^x\!\sin z(x\!-\!t)\cdot V^\sharp(t)\sin zt\, dt +
O\lt(\frac{e^{|\Im z|}}{|z|^4}\rt)
\end{array}
\]
uniformly on bounded subsets of $V$. Substituting $z^2=\pi^2n^2+\m$, $|\m|=3\|V\|=O(1)$, one
obtains
\[
\begin{array}{l} \displaystyle
\c(0,\pi^2n^2\!+\!\m,V)= \frac{(-1)^n\mu}{2\pi^2n^2}\,I_N+ \frac{1}{\pi^2n^2}\lt[\int_0^1 \sin
\pi n (1\!-\!t)\sin \pi nt\cdot V^\sharp(t)dt
\vphantom{\int_\big|}\cr\displaystyle\hphantom{\c(0,\pi^2n^2\!+\!\m}
+\frac{\m}{2\pi{}n}\int_0^1((1\!-\!t)\cos \pi n(1\!-\!t)\sin \pi nt + t\sin\pi n (1\!-\!t)\cos
\pi nt)\cdot V^\sharp(t)dt\rt]
\vphantom{\int_\big|}\cr\displaystyle\hphantom{\c(0,\pi^2n^2\!+\!\m}
+\frac{1}{\pi^3n^3}\int_0^1dx\int_0^x\sin \pi
n(1\!-\!x)\sin\pi{}n(x\!-\!t)\sin\pi{}nt\cdot{}V^\sharp(x)V^\sharp(t)dt
\vphantom{\int_\big|}\cr\displaystyle\hphantom{\c(0,\pi^2n^2\!+\!\m} + O\lt(\frac{1}{n^4}\rt)\
=\ \frac{(-1)^n}{2\pi^2n^2}\lt[\m K_n + L_n + O\lt(\frac{1}{n^2}\rt)\rt],
\end{array}
\]
where the matrices
\[
K_n = I_N+\frac{1}{2\pi n}\,\nh{[(1\!-\!2t)V^\sharp]}{sn} =
I_N+\frac{1}{2\pi{}n}\,\nh{[(1\!-\!2t)V]}{sn},
\]
\[
L_n = -\nh{V^\sharp}{0}+\nh{V^\sharp}{cn}+O\lt(\frac{1}{n}\rt) = -\nh{V}{0}
+\nh{V}{cn}+O\lt(\frac{1}{n}\rt)
\]
do not depend on $\m$. Hence, if $\m=3\|V\|$ and $n$ is sufficiently large, then
\[
\frac{(-1)^n}{2\pi^2n^2}\,[\c(0,\pi^2n^2\!+\!\m,V)]^{-1} = \left[\m
K_n+L_n\right]^{-1}+O\lt(\frac{1}{n^2}\rt).
\]
Also, note that
\[
\c'(0,z^2,V) = -\vp'(1,z^2,V^\sharp) = -{\cos z}\,I_N - \frac{1}{z}\int_0^1 \cos z(1-t)\cdot
V^\sharp(t)\sin zt\, dt + O\lt(\frac{e^{|\Im z|}}{|z|^2}\rt).
\]
Therefore,
\[
\c'(0,\pi^2n^2\!+\!\m,V)=
(-1)^{n-1}\lt[I_N-\frac{1}{2\pi n}\,\nh{V}{sn} +O\lt(\frac{1}{n^2}\rt)\rt]
\]
and
\[
-\frac{1}{2\pi^2n^2}\,[\c'\c^{-1}](0,\pi^2n^2\!+\!\m,V) =\lt[I_N-\frac{1}{2\pi
n}\,\nh{V}{sn}\rt]K_n^{-1}\left[\m I_N+L_nK_n^{-1}\right]^{-1}+O\lt(\frac{1}{n^2}\rt).
\]
Since $L_nK_n^{-1}$ doesn't depend on $\m$ and $3\|V\|=|\m|>|L_nK_n^{-1}|$ for sufficiently
large $n$, we have
\[
\frac{1}{2\pi i}\oint_{|\m|=3\|V\|}\left[\m I_N+L_nK_n^{-1}\right]^{-1}d\m = I_N,
\]
and so
\[
\frac{1}{2\pi^2n^2}\,B_n= \lt[I_N-\frac{1}{2\pi n}\,\nh{V}{sn}\rt]K_n^{-1} +
O\lt(\frac{1}{n^2}\rt)=I_N-\frac{1}{\pi n}\,\nh{[(1\!-\!t)V]}{sn} +
O\lt(\frac{1}{n^2}\rt).\qedhere
\]
\end{proof}

\smallskip

\subsection{Proof of the direct part in Theorem \ref{MainThm}}
\label{DirectProblem}
\begin{proof} In fact, all needed asymptotics have been obtained in Sect. \ref{SectAsymptRough},
\ref{SectAsymptFine}. First, asymptotics of the eigenvalues and the individual projectors have
been derived in Proposition \nolinebreak \ref{RoughAsymptEV}. Second, asymptotics of the
norming constants and the averaged projectors follows from Proposition \ref{AsymptBn} and
Lemma \ref{EquivAsymptB}. In order to prove (C) suppose that $\x:\C\to\C^N$ is some entire
vector-valued function such that $P_\a\xi(\lambda_\a)\!=\!0$ for all $\a\!\ge\! 1$,
$\xi(\lambda)\!=\!O(e^{|\Im\sqrt\lambda|})$ as $|\lambda|\to\infty$ and $\xi\in \cL^2(\R_+)$.
Due to Lemma 2.2 \cite{CK},
\[
[\c(0,\l,V)]^{-1}= [\vp^*(1,\ol{\l},V)]^{-1}=
(Z_\a^{-1}+O(\l\!-\!\l_\a))((\l\!-\!\l_\a)^{-1}P_\a + P_\a^\perp)\ \ \mathrm{as}\ \ \l\to\l_\a
\]
for some $Z_\a$ such that $\det Z_\a\ne 0$. Hence, the (vector-valued) function
\[
\o(\l)=[\c(0,\l,V)]^{-1}\x(\l)\,
\]
is entire. It follows from (\ref{xVpasympt}) that
\[
\o(\l)=O(|\l|^{1/2})\quad \mathrm{as}\quad \textstyle |\l|=\pi^2(n+\frac{1}{2})^2\to\infty.
\]
Thus, the Liouville Theorem gives $\o(\l)\equiv\o(0)=\o_0\in\C^N$ and $\x(\l)\equiv
\c(0,\l,V)\o_0$. If $\o_0\ne 0$, then this contradicts to $\xi\in \cL^2(\R_+)$ in view of
asymptotics (\ref{xVpasympt}).
\end{proof}

\subsection{Explicit formula for the Weyl-Titchmarsh function.}
\label{SectMFormula}

In this Sect.~we prove that the Weyl-Titchmarsh function $M(\l,V)$ can be written as the
regularized sum over all its poles. In other words, we give the explicit formula for $M(\l,V)$
involving only the spectral data $\l_\a(V)$ and $B_\a(V)= -\res{}_{\l=\l_\a}M(\l,V)$. The
proof is quite standard.

\begin{proposition}
\label{Mformula} Let $V\!=\!V^*\in \cL^2([0,1];\C^{N\ts N})$ satisfy (\ref{Vnondeg}). Then
\begin{equation}
\label{xM=}
\begin{array}{l}
\displaystyle M(\l) + \sum_{j=1}^N \sqrt{\l\!-\!v_j^0}\,\cot\sqrt{\l\!-\!v_j^0}\cdot P_j^0
\vphantom{\int_\big|}\cr\displaystyle
\vphantom{M(\l)} =\lt[\sum_{\a=1}^{\a^\dia}\frac{B_\a}{\l_\a\!-\!\l} -
\sum_{n=1}^{n^\dia-1}\sum_{j=1}^N\frac{2\pi^2n^2P_j^0}{\pi^2n^2\!+\!v_j^0\!-\!\l}\rt] +
\sum_{n=n^\dia}^{+\iy}\sum_{j=1}^N
\lt[\frac{B_{n,j}}{\l_{n,j}-\!\l}-\frac{2\pi^2n^2P_j^0}{\pi^2n^2\!+\!v_j^0\!-\!\l}\rt].
\end{array}
\end{equation}
The series converge uniformly on compact subsets of $\C$ that do not contain poles.
\end{proposition}

\begin{proof}
Note that
\[
\begin{array}{l}
\displaystyle
D_{n,j}(\l)=\frac{B_{n,j}}{\l_{n,j}-\!\l}-\frac{2\pi^2n^2P_j^0}{\pi^2n^2\!+\!v_j^0\!-\!\l}
\vphantom{\int_\big|}\cr\displaystyle\hphantom{D_{n,j}(\l)}  =
\frac{B_{n,j}-2\pi^2n^2P_j^0}{\pi^2n^2-\l} -
\frac{v_j^0(B_{n,j}-2\pi^2n^2P_j^0)}{(\pi^2n^2\!-\!\l)(\pi^2n^2\!+\!v_j^0\!-\!\l)} -
\frac{(\l_{n,j}\!-\!\pi^2n^2\!-\!v_j^0)B_{n,j}}{(\l_{n,j}\!-\!\l)(\pi^2n^2\!+\!v_j^0\!-\!\l)}\,.
\end{array}
\]
Due to Proposition \ref{AsymptBn}, for the first terms one has
\[
D_n^{(1)}(\l)=\sum_{j=1}^N\frac{B_{n,j}-2\pi^2n^2P_j^0}{\pi^2n^2-\l}= \frac{\sum_{j=1}^N
B_{n,j}-2\pi^2n^2I_N}{\pi^2n^2\!-\!\l}= \frac{n\cdot x_n}{\pi^2n^2\!-\!\l},
\]
where $(x_n)_{n=n^\dia}^{+\iy}\in\ell^2$. In particular, the series
$\sum_{n=n^\dia}^{+\iy}D_n^{(1)}(\l)$ uniformly converges outside singularities. Moreover,
\[
\lt|\sum_{n=n^\dia}^{+\iy}D_n^{(1)}(\l)\rt| \le \frac{1}{\pi^2}\sum_{n=n^\dia}^{+\iy}
\frac{|x_n|}{|n\!-\!(m\!+\!\frac{1}{2})|} \to 0\quad \mathrm{as}\ \
|\l|=\pi^2(m\!+\!{\textstyle\frac{1}{2}})^2\to\iy.
\]
Since $B_{n,j}=2\pi^2n^2(P_j^0+\ell^2)$ and $\l_{n,j}=\pi^2n^2+v_j^0+\ell^2$, the similar
results hold true for the sums of second and third terms of $D_{n,j}(\l)$.

Thus, the right-hand side of (\ref{xM=}) converges outside singularities and tends to zero as
$|\l|=\pi^2n^2(m\!+\!\frac{1}{2})^2\to\iy$. It follows from the standard asymptotics of
fundamental solutions that the left-hand side of (\ref{xM=}) also tends to zero as
$|\l|=\pi^2n^2(m\!+\!\frac{1}{2})^2\to\iy$. Since the residues of both sides at singularities
coincide, (\ref{xM=}) holds true for all $\l$.
\end{proof}

\section{Inverse problem}
\setcounter{equation}{0} \label{SectInverseProblem}

\subsection{Proof of the surjection part in Theorem \ref{MainThm}. General strategy.} $\vphantom{x}$
\label{SectGeneralStrategy}

\smallskip

\noindent {\bf Step 1.} Let some data $(\l_\a^\dia,P_\a^\dag,g_\a^\dag)_{\a\ge 1}$ satisfy
conditions (A)--(C) in Theorem \ref{MainThm} and $B_\a^\dag=P_\a^\dag (g_\a^\dag)^{-1}
P_\a^\dag$ (we use different superscript $\dia$ for eigenvalues in order to make the further
presentation more clear). Consider eigenvalues $\l_\a^\dia$ (possibly multiple for several
first $\a$). One can split them into $N$ {\it simple} series $\{\l_{n,j}^\dia\}_{n=1}^\infty$,
$j=1,2,..,N$ such that
\[
\{\l_{n,1}^\dia\}_{n=1}^{+\infty}\cup\{\l_{n,2}^\dia\}_{n=1}^{+\infty}\cup..\cup\{\l_{n,N}^\dia\}_{n=1}^{+\infty}
= \{\l_\a^\dia\}_{\a\ge 1}
\]
(counting with multiplicities) and $\l_{n,j}^\dia=\pi^2n^2+v_j^0+\ell^2$ for all $j=1,2,..,N$.

Using the well known scalar inverse theory (see (\ref{CondScalar})) we construct some scalar
potentials $v_{jj}^\dia\in \cL^2([0,1])$ such that
\[
\int_0^1 v_{jj}^\dia(t)dt=v_j^0\qquad \mathrm{and}\qquad
\s(v_{jj}^\dia)=\{\l_{n,j}^\dia\}_{n=1}^{+\infty}\,.
\]
Note that the corresponding isospectral sets are infinite dimensional manifolds, so there are
infinitely many choices for each $v_{jj}^\dia$. For technical reasons, {\it we choose
$v_{jj}^\dia$ such that}
\[
g_n^{-1}(v_{jj}^\dia)= -\res_{\l=\l_{n,j}}m(\l,v_{jj}^\dia)=2\pi^2n^2\quad \mathrm{for\ all\
sufficiently\ large}\ n,
\]
where $m(\l,v_{jj}^\dia)$ is the Weyl-Titchmarsh function of the scalar potential
$v_{jj}^\dia$, and
\[
\c'(0,\l_\a,v_{jj}^\dia)\ne 0,\quad \mathrm{i.e.,}\ \ m(\l_\a,v_{jj}^\dia)\ne 0,\quad
\mathrm{for\ all}\ \a\ge 1.
\]
(one can always choose such $v_{jj}^\dia$ in two steps: taking the scalar $m$-function with
all residues equal to $-2\pi^2n^2$ and changing the first residue slightly in order to
guarantee $m(\l_\a,v_{jj}^\dia)\ne 0$ for all $\a\ge 1$). Let
\[
V^\dia=\diag\{v_{11}^\dia,v_{22}^\dia,..,v_{NN}^\dia\}.
\]
Thus, $\s(V^\dia)=\{\l_\a^\dia\}_{\a\ge 1}$ counting with multiplicities. Denote
\[
B_\a^\dia=p_\a^\dia (g_\a^\dia)^{-1} (p_\a^\dia)^*=B_\a(V^\dia).
\]
Since $V^\dia$ is a diagonal potential, each subspace $\cE_\a^\dia$ is spanned by some (one,
if $\a$ is large enough) standard coordinate vectors $e_j^0$ and all $P_\a^\dia$ are
coordinate projectors.

\medskip

\noindent {\bf Step 2.} Let
\[
A_\a(V) = M^{-1}(\l_\a^\dia) = [\c(\c')^{-1}](0,\l_\a^\dia,V)
\]
and
\begin{equation} \label{A11def}
\begin{array}{ll}
\displaystyle \vphantom{\big|_|} A_\a^{11}=p_\a^\dia A_\a
(p_\a^\dia)^*:\cE_\a^\dia\to\cE_\a^\dia, & \displaystyle A_\a^{12}=p_\a^\dia A_\a
(q_\a^\dia)^*:(\cE_\a^\dia)^\perp\to\cE_\a^\dia, \cr \displaystyle \vphantom{\big|^|}
A_\a^{21}=q_\a^\dia A_\a (p_\a^\dia)^*:\cE_\a^\dia\to (\cE_\a^\dia)^\perp, & A_\a^{22}=
q_\a^\dia A_\a (q_\a^\dia)^*: (\cE_\a^\dia)^\perp \to (\cE_\a^\dia)^\perp,
\end{array}
\end{equation}
where $p_\a^\dia:\C^N\to\cE_\a^\dia$, $q_\a^\dia:\C^N\to(\cE_\a^\dia)^\perp$ are the
coordinate projectors.  Note that
\[
A_\a^{11}(V^\dia)=0,\ \ A_\a^{12}(V^\dia)=0,\ \ A_\a^{21}(V^\dia)=0\ \ \mathrm{and}\ \ \det
A_\a^{22}(V^\dia)\ne 0\ \ \mathrm{for\ all}\ \a\ge 1
\]
due to $p_\a^\dia \c(0,\l_\a^\dia,V^\dia)= [\vp(1,\l_\a^\dia,V^\dia)(p_\a^\dia)^*]^*=0$ and
$\det \c'(0,\l_\a^\dia,V^\dia)\ne 0$.

\smallskip

In order to describe some neighborhood of the isospectral set $\Iso(V^\dia)$ near $V^\dia$, we
introduce $k_\a\ts k_\a$ matrices (more accurate, operators in the coordinate subspaces
$\cE_\a^\dia$)
\begin{equation}
\label{aDef} \wt{A}_\a(V) = \left[A_\a^{11}-
A_\a^{12}(A_\a^{22})^{-1}A_\a^{21}\right](V),\quad \a\ge 1.
\end{equation}
Then (see Proposition \ref{aBnjRoughAsympt} and Lemma \ref{ABselfadjoint})

\smallskip

(i) all $\wt{A}_\a(V)$ are well-defined in some {\it complex} neighborhood
$\cB(V^\dia,r^\dia)$ of $V^\dia$;

(ii) for $V=V^*\in \cB(V^\dia,r^\dia)$ one has $\wt{A}_\a(V)=[\wt{A}_\a(V)]^*$ and the
following holds:
\begin{center}
{\it $\wt{A}_\a(V)=0$ iff $\l_\a^\dia$ is an eigenvalue of $V$ of multiplicity $k_\a$.}
\end{center}

\smallskip

\noindent Furthermore, for potentials $V$ sufficiently close to $V^\dia$, we set
\begin{equation}
\label{wtBdef} \wt{B}_\a(V)= -\frac{1}{2\pi i}\oint_{|\l-\l_\a^\dia|=d^\dia}M(\l,V)d\l,
\quad\mathrm{where}\quad d^\dia=\textstyle\frac{1}{2}\min_{\a\ge
1}(\l_{\a+1}^\dia\!-\!\l_\a^\dia)>0.
\end{equation}
If $k_\a^\dia=1$, then $M(\l)$ has exactly one simple pole inside this contour, so
\mbox{$\wt{B}_\a(V)\!=\!B_\a(V)$}. If $k_\a^\dia>1$, we do not know precisely how the multiple
eigenvalue $\l_\a^\dia$ is split, so $\wt{B}_\a(V)$ denotes the sum of all corresponding
residues. Then (see Proposition \ref{aBnjRoughAsympt}, Lemma \ref{ABselfadjoint})

\smallskip

(i) all $\wt{B}_\a(V)$ are well-defined in some {\it complex} neighborhood
$\cB(V^\dia,r^\dia)$ of $V^\dia$;

(ii) for $V\!=\!V^*\!\in\!\cB(V^\dia,r^\dia)$ one has $\wt{B}_\a\!=\!\wt{B}_\a^*$, $\rank
\wt{B}_\a\!=\!k_\a^\dia$ and the following holds:
\[
\wt{A}_\a(V)=0\quad \Rightarrow\quad \wt{B}_\a(V)=B_\a(V).
\]
In other words, $\wt{B}_\a(V)$ is the analytic continuation of $B_\a(V)$ from the isospectral
set $\Iso (V^\dia)$ into some complex neighborhood of $V^\dia$ (emphasize that, due to the
\mbox{possible} splitting of the eigenvalue $\l_\a^\dia$ in case $k_\a^\dia>1$, the original
function $B_\a(V)$ is {\it discontinuous} even for self-adjoint potentials close to $V^\dia$).

\medskip

\noindent {\bf Step 3.} We introduce the mapping
\[
\wt{\F}:V\mapsto (\wt{A}_\a(V);\wt{B}_\a(V))_{\a\ge 1}
\]
which is defined in some complex neighborhood $\cB(V^\dia,r^\dia)$ of $V^\dia$ (see
\mbox{Sect. \ref{SectRoughAsymptAB}}). We prove that $\wt{\F}$ maps $\cB(V^\dia,r^\dia)$ into
some "proper" $\ell^2$-type space. In order to have the "nice" description of the image space,
we consider some {\it modification} $\F$, see details in Sect. \ref{SectAnalyticity1},
\ref{SectAnalyticity2}. The modified mapping $\F$ is analytic in $\cB(V^\dia,r^\dia)$, so its
restriction onto self-adjoint potentials close to $V^\dia$ is real-analytic. Note that, if
$V=V^*$, then both $k_\a^\dia\!\ts\!k_\a^\dia$ matrix $\wt{A}_\a$ and $N\!\ts\!N$ matrix
$\wt{B}_\a$, \mbox{$\rank\wt{B}_\a=k_\a^\dia$}, are self-adjoint. So, the total number of
(real) parameters in $(\wt{A}_\a(V),\wt{B}_\a(V))$ is
$(k_\a^\dia)^2+k_\a^\dia(2N\!-\!k_\a^\dia)=2Nk_\a^\dia$.

\medskip

\noindent {\bf Step 4} We check that the Fr\'echet derivative $d_{V^\dia}\F$ of the modified
mapping $\F$ at the point $V^\dia$ is invertible (see details in Sect. \ref{SectFrechetDer1},
\ref{SectFrechetDer2}) . Therefore, due to the Implicit Function Theorem, for each sequence
$(B_\a^\bullet)_{\a\ge 1}$ sufficiently close to $(B_\a^\dia)_{\a\ge 1}$ there exists some
potential $V^\bullet$ (close to $V^\dia$) such that $\wt{A}_\a(V^\bullet)=\wt{A}_\a(V^\dia)=0$
and $\wt{B}_\a(V^\bullet)=B_\a^\bullet$ for all $\a\ge 1$. If $\a^\bullet$ is large enough,
then the sequence
\[
B_\a^\bullet := B_\a^\dia,\ \ \mathrm{if}\ \a\le\a^\bullet,\quad  \mathrm{and}\quad
B_\a^\bullet:=B_\a^\dag,\ \ \mathrm{if}\ \a>\a^\bullet,
\]
is close to $(B_\a^\dia)_{\a\ge 1}$. Thus, we obtain some potential $V^\bullet$ such that
\[
\wt{A}_\a(V^\bullet)= 0 \quad \mathrm{for\ all}\ \a\ge 1,\qquad \mathrm{i.e.,}\quad
\s(V^\bullet)=\{\l_\a^\dia\}_{\a\ge 1}
\]
(counting with multiplicities) and
\[
B_\a(V^\bullet)=\wt{B}_\a(V^\bullet)=B_\a^\dag,\quad \mathrm{for}\ \a>\a^\bullet.
\]
Finally, using the isospectral transforms constructed in \cite{CK}, we change the {\it finite}
number of residues $B_\a$, $\a=1,2,..,\a^\bullet$  (see details in Sect.
\ref{SectChangingFirst}), and obtain the potential having the given spectral data
$(\l_\a^\dia,B_\a^\dag)_{\a\ge 1}$ or, equivalently, $(\l_\a^\dia,P_\a^\dag,g_\a^\dag)_{\a\ge
1}$. \qed

\subsection{Rough asymptotics of $\bf \wt{A}_\a(V)$ and $\bf \wt B_\a(V)$.}

\label{SectRoughAsymptAB}

This section contains some preliminary calculations. Loosely speaking, we consider the
diagonal potential $V^\dia$ as the unperturbed case and derive some rough asymptotics of
spectral data for $V$ close to~$V^\dia$. The main results are formulated in \mbox{Proposition
\ref{aBnjRoughAsympt}} and Lemma \ref{ABselfadjoint}.

Let $\vp^\dia,\vt^\dia,\c^\dia,\e^\dia$ be the standard {\it diagonal} matrix-valued solutions
(recall that $V^\dia$ is diagonal) of the equation $-\p''(x)+V^\dia(x)\p(x)=\l\p(x)$
satisfying the following boundary conditions:
\[
\begin{array}{l}
\displaystyle \vt^\dia(0)=(\vp^\dia)'(0)=I_N, \cr \displaystyle (\vt^\dia)'(0)=\vp^\dia(0)=0,
\end{array}
\qquad
\begin{array}{l}
\displaystyle \e^\dia(1)=-(\c^\dia)'(1)=I_N, \cr \displaystyle (\e^\dia)'(1)=\c^\dia(1)=0.
\end{array}
\]
We denote $\vp^\dia_\a(x)=\vp^\dia(x,\l_\a^\dia)$, $\vt^\dia_\a(x)=\vt(x,\l_\a^\dia)$ and so
on. Let
\[
J^\dia(x,t)= \vp^\dia(x)\vt^\dia(t)-\vt^\dia(x)\vp^\dia(t)=
-\c^\dia(x)\e^\dia(t)+\e^\dia(x)\c^\dia(t)
\]
be the (diagonal) solution of the same equation such that $J^\dia(t,t)=0$, $
(J^\dia)'_x(t,t)=I_N$.

Let $V=V^\dia+W$ be some complex potential close to $V^\dia$. Then $\c(x,\l,V)$ can be easily
constructed by iterations with the kernel $J^\dia(x,t)$ (note that
$|J^\dia(x,t;z^2)|=O(|z|^{-1}{e^{|\Im z|\cdot|x-t|}})$) starting with $\c^\dia(x,\l)$ . Thus,
\begin{equation}
\label{ChiW=} \c(0,z^2,V)=\c^\dia(0,z^2)+\int_0^1\vp^\dia(t,z^2)W(t)\c^\dia(t,z^2)dt+
O\lt(\frac{\|W\|^2e^{|\Im z|}}{|z|^3}\rt),
\end{equation}
\begin{equation}
\label{Chi'W=} \c'(0,z^2,V)=(\c^\dia)'(0,z^2)-\int_0^1\vt^\dia(t,z^2)W(t)\c^\dia(t,z^2)dt+
O\lt(\frac{\|W\|^2e^{|\Im z|}}{|z|^2}\rt)
\end{equation}
uniformly on bounded subsets of $W$. In particular (see \ref{Vp1Asympt}), if $\m=O(1)$, then
\begin{equation}
\label{ChiDia0=}
\begin{array}{rcl}
\c(0,\l_{n,j}^\dia\!+\!\m,V) &\!\!\!=\!\!\!&
\displaystyle\frac{(-1)^n}{2\pi^2n^2}\left(\vphantom{|^|_|}\diag\{\m\!-\!v_j^0\!+\!v_1^0, ..,
\m\!-\!v_j^0\!+\!v_N^0\}\!+\!o(1)\!+\!O(\|W\|)\right),\cr \displaystyle\vphantom{|^\big|}
\c'(0,\l_{n,j}^\dia\!+\!\m,V) &\!\!\!=\!\!\!& (-1)^{n-1}\!\left(I_N+O(n^{-1})\right), \qquad
\mathrm{as}\ \ n\to\iy,
\end{array}
\end{equation}
uniformly on bounded subsets of $W$. Recall that $A_\a(V)\!=\![\c(\c')^{-1}](0,\l_\a^\dia,V)$
and its block $A_\a^{22}=q_\a^\dia A_\a(q_\a^\dia)^*$ are given by (\ref{A11def}) and
$d^\dia=\frac{1}{2}\min_{\a\ge 1}(\l_{\a+1}^\dia\!-\!\l_\a^\dia)>0$.

\begin{lemma}
\label{xABwell-def} There exists $r^\dia>0$ such that for all (possibly non-selfadjoint)
potentials
\[
V\in \cB(V^\dia,r^\dia)= \left\{V\in \cL^2([0,1];\C^{N\ts N}):\|V\!-\!V^\dia\|<r^\dia\right\}
\]
the following is fulfilled for all $\a\ge 1$:
\[
\det\c'(0,\l_\a^\dia,V)\ne 0,\ \ \det A_\a^{22}(V)\ne 0\ \ \mathit{and}\ \
\det\c(0,\l_\a^\dia\!+\!\m,V)\ne 0,\ \mathit{if}\ |\m|=d^\dia.
\]
Moreover, for all $j=1,2,..,N$ and $|\m|=d^\dia$,
\begin{equation}
\label{xInv} \left[A_{n,j}^{22}(V)\right]^{-1}\!= O(n^2)\quad \mathit{and}\quad
\left[\c(0,\l_{n,j}^\dia\!+\!\m,V)\right]^{-1}\!= O(n^2)
\end{equation}
uniformly on $\cB(V^\dia,r^\dia)$.
\end{lemma}

\begin{proof}
It follows from (\ref{ChiDia0=}) that all matrices $\c'(0,\l_{n,j}^\dia,V)$,
$A_{n,j}^{22}(V)$, $\c(0,\l_{n,j}^\dia\!+\!\m,V)$ are non-degenerate and (\ref{xInv}) holds,
if $n\ge n_*$ is sufficiently large and $r^\dia$ is sufficiently small. So, one needs to
consider only some finite number of first indices $\a=1,2,..,\a_*$.

Note that $\det\c'(0,\l_\a^\dia,V^\dia)\ne 0$, $\det A_\a^{22}(V^\dia)\ne 0$,
$\det\c(0,\l_\a^\dia\!+\!\m,V^\dia)\ne 0$ for all $\a$ and all these matrices (as functions of
$V$) are continuous at $V^\dia$. Therefore, if $\|W\|\le r^\dia$ and $r^\dia>0$ is small
enough, then all $\c'(0,\l_\a^\dia,V)$, $A_\a^{22}(V)$, $\c(0,\l_\a^\dia+\m,V)$,
$\a=1,2,..,\a_*$, are non-degenerate too.
\end{proof}

\begin{proposition}
\label{aBnjRoughAsympt} (i) There exists $r^\dia\!>\!0$ such that all $\wt{A}_\a(V)$,
$\wt{B}_\a(V)$, $\a\ge 1$, are well-defined by (\ref{aDef}), (\ref{wtBdef}) and analytic in
$\cB(V^\dia,r^\dia)$.

\smallskip

\noindent (ii) For all $j=1,2,..,N$ the asymptotics
\[
\wt{A}_{n,j}(V)= O\lt(\frac{\ve_n(W)}{n^2}\rt),\quad \wt{B}_{n,j}(V)-B_{n,j}^\dia =
O\left(n^2\ve_n(W)\right),\quad\ve_n(W)=|\nh W{cn}|+ \frac{\|W\|}{n},
\]
hold true uniformly for potentials
\[
V\in\cB^0(V^\dia,r^\dia)=\lt\{V=V^\dia\!+\!W\in \cB(V^\dia,r^\dia): \int_0^1W(t)dt=0\rt\}.
\]
\end{proposition}
\begin{proof}
(i) Due to Lemma \ref{xABwell-def}, all $\wt{A}_\a(V)$, $\wt{B}_\a(V)$ are well-defined in
some complex neighborhood $\cB(V^\dia,r^\dia)$ of $V^\dia$. These functions are analytic in
this neighborhood since $\c(0,\l,V)$ and $\c'(0,\l,V)$ are analytic for each $\l$ as functions
of $V$.

\smallskip

\noindent (ii) Let $\l=\pi^2n^2\!+\!\m$ and $|\m|=O(1)$, thus
\[
\vp^\dia(t,\l)= (\pi n)^{-1}\sin\pi nt + O(n^{-2})\ \ \mathrm{and}\ \
(-1)^{n-1}\c^\dia(t,\l)=(\pi n)^{-1}{\sin\pi nt}+O(n^{-2}).
\]
Using (\ref{ChiW=}), (\ref{Chi'W=}) and $\int_0^1W(t)dt=0$, we~get
\[
\c(0,\l,V)=\c^\dia(0,\l)+ O\lt(\frac{\ve_n(W)}{n^2}\rt),\quad
\c'(0,\l,V)=(\c^\dia)'(0,\l)+ O\lt(\frac{\|W\|}{n}\rt)
\]
(note that $n^{-1}\|W\|\le \ve_n(W)$ by definition). Due to (\ref{ChiDia0=}), it gives
\[
A_{n,j}(V)= [\c(\c')^{-1}](0,\l_{n,j}^\dia,V)=A_{n,j}(V^\dia) + O\lt(\frac{\ve_n(W)}{n^2}\rt).
\]
Since $A_{n,j}^{11}(V^\dia)\!=\!0$, $A_{n,j}^{12}(V^\dia)\!=\!0$, $A_{n,j}^{21}(V^\dia)\!=\!0$
and $(A_{n,j}^{22}(V))^{-1}\!=\!O(n^2)$, we have
\[
\wt{A}_{n,j}(V)= \left[A_{n,j}^{11} - A_{n,j}^{12}(A_{n,j}^{22})^{-1}A_{n,j}^{21}\right](V)=
O\lt(\frac{\ve_n(W)}{n^2}\rt).
\]
Due to the similar arguments, if $\l=\l_{n,j}^\dia\!+\!\m$, $|\m|=d^\dia$, then
\[
[\c'\c^{-1}](0,\l,V) = [(\c^\dia)'(\c^\dia)^{-1}](0,\l)+ O\left(n^2\ve_n(W)\right).
\]
Integrating over the contour $|\m|=d^\dia$, we obtain $\wt{B}_{n,j}(V)=B_{n,j}^\dia +
O\left(n^2\ve_n(W)\right)$.
\end{proof}

\begin{lemma}
\label{ABselfadjoint} For some $r^\dia>0$ and all $V=V^*\in\cB(V^\dia,r^\dia)$ the following
hold:

\smallskip

(i) $\wt{A}_\a(V)=[\wt{A}_\a(V)]^*$, $\wt{B}_\a(V)=[\wt{B}_\a(V)]^*$ and $\rank
\wt{B}_\a(V)=k_\a^\dia$;

\smallskip

(ii) $\wt{A}_\a(V)=0$ if and only if $\bf\l_\a^\dia$ is an eigenvalue of $V$ of multiplicity
$k_\a^\dia$;

\smallskip

(iii) if $\wt{A}_\a(V)=0$, then $\wt{B}_\a(V)=B_\a(V)$.
\end{lemma}

\begin{proof}
(i) If $V=V^*$, then $M(\l)\equiv [M(\ol{\l})]^*$, $\l\in\C$. In particular,
$\wt{B}_\a(V)=[\wt{B}_\a(V)]^*$, $A_\a(V)=[A_\a(V)]^*$ and $\wt{A}_\a(V)=[\wt{A}_\a(V)]^*$.
Due to Lemma \ref{xABwell-def}, $\det\c(0,\l,V)$ has no zeros on the circle
$|\l-\l_\a|=d^\dia$ for all $V\in\cB(V^\dia,r^\dia)$. Since the spectrum depends on the
potentials continuously, for each self-adjoint potential $V=V^*\in\cB(V^\dia,r^\dia)$ there
are exactly $k_\a^\dia$ eigenvalues in the interval
$(\l_\a^\dia\!-\!d^\dia,\l_\a^\dia\!+\!d^\dia)$ counting with multiplicities.

If $\a>\a^\dia$, then $k_\a^\dia=1$ and $\rank \wt{B}_\a(V)=\rank B_\a(V)=1$. If
$\a\le\a^\dia$, then $\rank \wt{B}_\a(V)\le k_\a^\dia$. Note that $\rank \wt{B}_\a(V^\dia)=
k_\a^\dia$ and $\wt{B}_\a$ is a continuous function of $V$. Thus, if $r^\dia$ is small enough,
then $\rank \wt{B}_\a(V)\ge k_\a^\dia$ for all $\a\le\a^\dia$ and $V\in\cB(V^\dia,r^\dia)$.

\smallskip

\noindent (ii) Recall that $\l_\a^\dia$ is an eigenvalue of $V$ of multiplicity $k_\a^\dia$
iff $\dim \Ker \c(0,\l_\a^\dia,V)=k_\a^\dia$. Since $\det\c'(0,\l_\a^\dia,V)\ne 0$ (see Lemma
\ref{xABwell-def}), this is equivalent to say that
\[
\dim \Ker [\c(\c')^{-1}](0,\l_\a^\dia,V)=k_\a^\dia,\quad \mathrm{i.e.,}\quad \rank
A_\a(V)=N-k_\a.
\]
Due to Lemma \ref{xABwell-def}, $\det A_\a^{22}(V)\ne 0$ for all $V\in\cB(V^\dia,r^\dia)$.
Then, the last statement is equivalent to $\wt{A}_\a(V)= [A_\a^{11}-
A_\a^{12}(A_\a^{22})^{-1}A_\a^{21}](V)=0$.

\smallskip

\noindent (iii) If $\wt{A}_\a(V)=0$, then $\l_\a^\dia$ is an eigenvalue of multiplicity
$k_\a^\dia$ and there are no other eigenvalues in the disc $|\l\!-\!\l_\a^\dia|<d^\dia$. Thus,
\[
\wt{B}_\a(V)= - \res{}_{\l=\l_\a^\dia} M(\l,V) = B_\a(V). \qedhere
\]
\end{proof}

\subsection{Analyticity. Expanded mapping $\bf \P$.}

\label{SectAnalyticity1}

Proposition \ref{aBnjRoughAsympt} (i) guarantees that all matrices \mbox{$\wt{A}_\a(V)$,
$\wt{B}_\a(V)$, $\a\ge 1$}, are well-defined in some neighborhood $\cB(V^\dia,r^\dia)$ of
$V^\dia$. Let $\a^\dia\ge 0$ and $n^\dia\ge 1$ be such that
\[
k_1^\dia+k_2^\dia+..+k_{\a^\dia}^\dia=N(n^\dia\!-\!1)\quad \mathrm {and} \quad k_\a^\dia=1\ \
\mathrm{for\ all}\ \a\ge\a^\dia\!+\!1,
\]
so the double-indexing $(n,j)$, $j=1,2,..,N$, is well-defined starting with $n^\dia$. Also,
let $n^\dia$ be sufficiently large such that $g_{n,j}^{-1}(V^\dia)=2\pi^2n^2$ for all $n\ge
n^\dia$ (see Step 1 Sect. \ref{SectGeneralStrategy}). Recall that
$B_n(V)=\sum_{j=1}^N\wt{B}_{n,j}(V)$ for $n\ge n^\dia$.
\begin{definition} \label{DefPsi}
Introduce the (formal) mapping
\[
\begin{array}{lcl}
\P: V & \mapsto & \displaystyle \left(\vphantom{\big|}\P^{(1)}(V)\ ; \P^{(2)}(V) \right) =
\lt(\left(\P^{(1)}_\a(V)\right)_{\a=1}^{\a^\dia}\ ;
\left(\P^{(2)}_n(V)\right)_{n=n^\dia}^{+\infty} \rt), \cr
  \P^{(1)}_\a & = & \displaystyle \left(\wt{A}_\a\ ;\wt{B}_\a\right), \cr
  \P^{(2)}_n & = & \displaystyle\lt(\left(2\pi^2n^2\cdot
\wt{A}_{n,j}\right)_{j=1}^N\ ;\ \lt(\frac{\wt{B}_{n,j}}{2\pi^2n^2}-P_j^0\rt)_{\!j=1}^{\!N}\ ;\
\pi n\lt[\frac{B_n}{2\pi^2n^2}-I_N\rt]\,\rt).
\end{array}
\]
\end{definition}

\noindent Note that $\P^{(1)}_\a$ and $\P^{(2)}_n$ map $\cB(V^\dia,r^\dia)$ into some
finite-dimensional spaces. Namely,
\[
\P^{(1)}_\a: \cB(V^\dia,r^\dia)\to \C^{k^\dia_\a\ts k^\dia_\a}\oplus \C^{N\ts N} \ \
\mathrm{and}\ \ \P^{(2)}_n: \cB(V^\dia,r^\dia)\to \C^N\oplus\!\left[\C^{N\ts N}\right]^N\!\!
\oplus \C^{N\ts N}.
\]
Since $\P^{(1)}$ has the finite number of components, it also acts into {\it
finite-dimensional} Hilbert (Euclidian) space $\wt{\cH}^{(1)} =
\Oplus_{\a=1}^{\a^\dia}\left[\C^{k^\dia_\a\!\ts k^\dia_\a}\!\oplus \C^{N\ts N}\right]$. It has
been shown in Sect.~\ref{SectRoughAsymptAB} that the components of $\P^{(2)}$ have "nice"
asymptotics for potentials
\[
V\in \cB^0(V^\dia,r^\dia)=\lt\{V=V^\dia\!+\!W\in \cB(V^\dia,r^\dia):  \int_0^1W(t)dt=0 \rt\}.
\]
Let $\N_{n^\dia}=\{n\!\in\!\N:n\!\ge\! n^\dia\}$ and $\C^{m\ts m}_\R=\{A\!=\!A^*\!\in\!
\C^{m\ts m}\}$ be the real component of the complex Hilbert space $\C^{m\ts m}$, i.e., the
real space of all self-adjoint $m\!\ts\!m$ matrices.
\begin{lemma}
\label{PsiAnalyticity} (i) $\P^{(2)}$ maps $\cB^0(V^\dia,r^\dia)$ into
$\wt{\cH}^{(2)}=\,\ell^2_\C\left(\,\N_{n^\dia}\,;\, \C^N\!\oplus\!\left[\C^{N\ts
N}\right]^N\!\!\oplus \C^{N\ts N} \right)$. Moreover, the image
$\P^{(2)}\left[\cB^0(V^\dia,r^\dia)\right]$ is bounded in $\wt{\cH}^{(2)}$.

\smallskip

\noindent (ii) $\P:\cB^0(V^\dia,r^\dia)\to \wt{\cH}=\wt{\cH}^{(1)}\oplus \wt{\cH}^{(2)}$ is an
analytic mapping between complex Hilbert spaces. Moreover, the Fr\'echet derivative
$d_{V^\dia}\P$ of $\P$ at $V^\dia$ is given by the Fr\'echet derivatives of its components:
$(d_{V^\dia}\P)W=\left(((d_{V^\dia}\P^{(1)}_\a)W)_{\a=1}^{\a^\dia}\ ;
((d_{V^\dia}\P^{(2)}_n)W)_{n=n^\dia}^{+\infty} \right)$.

\smallskip

\noindent (iii) $\P:\cB^0_\R(V^\dia,r^\dia)=\cB^0(V^\dia,r^\dia)\cap\cL^2([0,1];\C^{N\ts
N}_\R)\to \wt{\cH}_\R=\wt{\cH}^{(1)}_\R\ts \wt{\cH}^{(2)}_\R$ is a real-analytic mapping
between real Hilbert spaces and the Fr\'echet derivative $d_{V^\dia}\P$ is given by the
Fr\'echet derivatives of its components, where
\[
\wt{\cH}^{(1)}_\R = \Oplus_{\a=1}^{\a^\dia}\left[\C^{k^\dia_\a\!\ts k^\dia_\a}_\R\!\oplus
\C^{N\ts N}_\R\right], \qquad \wt{\cH}^{(2)}_\R\ =
\,\ell^2_\R\left(\,\N_{n^\dia}\,;\,\R^N\oplus\left[\C^{N\ts N}_\R\right]^N\!\!\oplus \C^{N\ts
N}_\R \right)\!.
\]
\end{lemma}

\begin{proof}
(i) Due to Proposition \ref{aBnjRoughAsympt}, for all $j=1,2,..,N$
\[
\wt{A}_{n,j}(V)= O(n^{-2}\ve_n(W))\quad \mathrm{and}\quad \wt{B}_{n,j}(V)- B_{n,j}^\dia =
O(n^2\ve_n(W))
\]
uniformly on $\cB(V^\dia,r^\dia)$, where
\[
\ve_n(W)=|\nh W{cn}|+ \frac{\|W\|}{n},\qquad \mathrm{so}\qquad \sum_{n=n^\dia}^{+\infty}
|\ve_n(W)|^2=O(\|W\|^2).
\]
Since $B_{n,j}^\dia=(g_{n,j}^\dia)^{-1}P_j^0=2\pi^2n^2P_j^0$, $n\ge n^\dia$, we obtain
\[
\left(2\pi^2n^2\wt{A}_{n,j}(V)\right)_{n=n^\dia}^{+\iy}\in \ell^2 \qquad \mathrm{and} \qquad
\lt(\frac{\wt{B}_{n,j}(V)}{2\pi^2n^2}-P_j^0\rt)_{\!n=n^\dia}^{+\iy}\in \ell^2,\ \ j=1,2,..,N,
\]
uniformly on $\cB^0(V^\dia,r^\dia)$. Also, due to Proposition \ref{AsymptBn},
\[
\lt(\pi n\lt[\frac{B_n(V)}{2\pi^2n^2}-I_N\rt]\rt)_{\!n=n^\dia}^{\!+\iy}\in \ell^2\quad
\mathrm{uniformly\ on}\ \cB^0(V^\dia,r^\dia).
\]

\noindent (ii) Due to Proposition \ref{aBnjRoughAsympt}, all coordinates $\P_\a^{(1)}$,
$\a=1,2,..,\a^\dia$, are analytic in $\cB(V^\dia,r^\dia)$. Hence, $\P^{(1)}$ is analytic too.
Similarly, all coordinates $\P_n^{(2)}$, $n\ge n^\dia$, are analytic in $\cB(V^\dia,r^\dia)$.
It follows from (i), that $\P^{(2)}$ is also locally bounded in $\cB^0(V^\dia,r^\dia)$.
Therefore (e.g., see \cite{PT} (Appendix A, Theorem 3) or \cite{Di} (Chapter 3, Proposition
3.7)), $\P^{(2)}$~is analytic as the mapping between Hilbert spaces and its Fr\'echet
derivative (or, equivalently, gradient) is given by the Fr\'echet derivatives (gradients) of
its components.

\smallskip

\noindent (iii) By Lemma \ref{ABselfadjoint}, $\P$ maps $B^0_\R(V^\dia, r^\dia)$ into
$\wt{\cH}_\R$. $\P$ is real-analytic due to (ii).
\end{proof}

\subsection{Analyticity. Modified mapping $\bf \F$.}

\label{SectAnalyticity2}

The expanded mapping $\P$ introduced in Definition \ref{DefPsi} is real-analytic but
overdetermined. In other words, its coordinates, obviously, are not independent from each
other. In particular, there are no chances that the Fr\'echet derivative $d_{V^\dia}\P$ is
invertible. On the other hand, the coordinates $\wt{A}_\a(V)$, $\wt{B}_\a(V)$,
\mbox{$\a\!\ge\!1$}, of the original mapping $\wt{\F}$ are independent, but we have no "nice"
description of the image space. The next goal is to construct some modified mapping
\mbox{$\F=(\F^{(1)},\F^{(2)})$} (see Definitions \ref{DefPhi1}, \ref{aceYUSdef},
\ref{DefPhi2}) such that
\begin{quotation}
\noindent (i) it keeps the full information about $\wt{A}_\a(V)$, $\wt{B}_\a(V)$, $\a\!\ge\!1$;\\
(ii) it is real-analytic as the mapping between Hilbert spaces;\\
(iii) its coordinates are "independent" from each other (more precisely, in Sect.
\ref{SectFrechetDer1}, \ref{SectFrechetDer2} we will show that $d_{V^\dia}\F$ is an invertible
linear operator).
\end{quotation}

We start with a slight modification of the first coordinates $\wt{B}_\a(V)$,
$\a=1,2,..,\a^\dia$. Recall that, if $V\in \cB^0_\R(V^\dia,r^\dia)$, then
$\wt{B}_\a(V)=[\wt{B}_\a(V)]^*$, $\rank \wt{B}_\a(V)=k^\dia_\a$ and
\[
B_\a^\dia=(p_\a^\dia)^* B_\a^\dia p_\a^\dia, \qquad p_\a^\dia B_\a^\dia(p_\a^\dia)^* =
(g_\a^\dia)^{-1} = [(g_\a^\dia)^{-1}]^*>0
\]
(moreover, $g_\a^\dia$ is diagonal, since $V^\dia$ is diagonal). Therefore, if $r^\dia\!>\!0$
is sufficiently small, then for each $\a=1,2,..,\a^\dia$ we have the (unique) factorization
\begin{equation}
\label{FactorCE} \wt{B}_\a = \left[(p_\a^\dia)^*+(q_\a^\dia)^*E_\a\right] C_\a
\left[p_\a^\dia+E_\a^*q_\a^\dia\right],\qquad
\begin{array}{l}
C_\a\,=\,C_\a^*\,=\wt{B}_\a^{11}: \cE_\a^\dia\to\cE_\a^\dia, \vphantom{|_\big|}\cr
E_\a=\wt{B}_\a^{21}[\wt{B}_\a^{11}]^{-1}: \cE_\a^\dia\to (\cE_\a^\dia)^\perp,
\end{array}
\end{equation}
where $\wt{B}_\a^{11}=p_\a^\dia\wt{B}_\a(p_\a^\dia)^*$,
$\wt{B}_\a^{21}=q_\a^\dia\wt{B}_\a(p_\a^\dia)^*$ etc. Note that $C_\a>0$, since $\rank
\wt{B}_\a = k_\a^\dia$.

\begin{definition}
\label{DefPhi1} We introduce the first component of the mapping $\F$ by
\begin{equation}
\begin{array}{lcl}
\displaystyle\vphantom{\big|_\big|} \F^{(1)}:\cB^0_\R(V^\dia,r^\dia)\ \to\ \cH^{(1)}_\R & = &
\Oplus_{\a=1}^{\a^\dia}\left[\C^{k^\dia_\a\!\ts k^\dia_\a}_\R\!\oplus \C^{k^\dia_\a\!\ts
k^\dia_\a}_\R\!\oplus \C^{(N-k^\dia_\a)\ts k^\dia_\a}_{\phantom\R}\right]\!, \cr
\displaystyle\vphantom{\big|^\big|}
\F^{(1)}(V)=\left(\F_\a^{(1)}(V)\right)_{\a=1}^{\a^\dia},&&
\F_\a^{(1)}(V)=\left(\wt{A}_\a(V)\,;\,C_\a(V)\,;\,E_\a(V)\right).\cr
\end{array}
\end{equation}
\end{definition}
\begin{remark}
Due to Lemma \ref{PsiAnalyticity} (ii), $\F^{(1)}$ is well-defined and real-analytic in
$\cB^0_\R(V^\dia,r^\dia)$, if $r^\dia\!>\!0$ is small enough. Note that $\wt{\F}^{(1)}$ can be
reconstructed from $\F^{(1)}$ and the total number of real parameters containing in $\F^{(1)}$
is $2N(k_1^\dia+k_2^\dia+..+k_{\a^\dia}^\dia)=2N^2(n^\dia\!-\!1)$.
\end{remark}

We pass to the design of the second component $\F^{(2)}$. The main purpose of (rather
technical) Definition \ref{aceYUSdef} is to combine heterogeneous objects from
(\ref{AsymptInThm}) into one object having "nice" asymptotics as $n\to\iy$ (see Proposition
\ref{PropF2}).

\smallskip

Due to Proposition \ref{aBnjRoughAsympt}, if $r^\dia\!>\!0$ is sufficiently small, then
\begin{equation}
\label{xABclose} |\wt{A}_{n,j}(V)|=O(n^{-2} \ve_n(W))\quad \mathrm{and} \quad |\wt{B}_{n,j}(V)
- 2\pi^2n^2 P_j^0|=O(n^2\ve_n(W)).
\end{equation}
In particular, if $V\in\cB^0_\R(V^\dia,r^\dia)$, then factorization (\ref{FactorCE}) is
well-defined for all $n\ge n^\dia$. Recall that $k_{n,j}^\dia=1$, so $\wt{A}_{n,j}(V)$ and
$C_{n,j}(V)>0$ are real numbers.
\begin{definition}
\label{aceYUSdef} Let $V\in\cB^0_\R(V^\dia,r^\dia)$ and $r^\dia>0$ be sufficiently small.
Introduce two numbers $a_{n,j}(V), c_{n,j}(V)\in\R$ and one vector $e_{n,j}(V)\in\C^N$ such
that $\langle e_{n,j},e_j^0\rangle =1$ as
\[
\vphantom{\Big|}a_{n,j}(V)=2\pi^2n^2\wt{A}_{n,j}(V),\quad
c_{n,j}(V)=\left[(2\pi^2n^2)^{-1}C_{n,j}(V)\right]^{\frac{1}{2}}\!,\quad e_{n,j}(V)= e_j^0 +
E_{n,j}(V)e_j^0.
\]
Furthermore, define $N\!\ts\!N$ matrix $Y_n=Y_n(V)\in\C^{N\ts N}$ by
\[
\vphantom{\Big|}Y_n=\left(\begin{array}{ccccccc}\exp[ia_{n,1}]\cdot c_{n,1}\cdot e_{n,1} &;&
\exp[ia_{n,2}]\cdot c_{n,2}\cdot e_{n,2} &; & ... & ; & \exp[ia_{n,N}]\cdot c_{n,N}\cdot
e_{n,N}\end{array}\right)
\]
and let
\[
Y_n(V) = U_n(V)S_n(V),\qquad U_n^*=U_n^{-1},\quad S_n^*=S_n>0,
\]
be its polar decomposition.
\end{definition}
Note that all $\wt{A}_{n,j}$, $\wt{B}_{n,j}$, $j=1,2,..,N$, can be easily reconstructed from
$U_n$, $S_n$. Factorization (\ref{FactorCE}) reads now as
\[
(2\pi^2n^2)^{-1}\wt{B}_{n,j}=c_{n,j}^2\cdot e_{n,j}e_{n,j}^*,
\]
so (\ref{xABclose}) gives
\[
|a_{n,j}(V)|\,,\,|c_{n,j}(V)-1|\,,\,|e_{n,j}(V) - e_j^0|=O(\ve_n(W))
\]
uniformly for $n\ge n^\dia$. Hence,
\begin{equation}
\label{xYUSclose} |Y_n(V)-I_N|\,,\,|U_n(V)-I_N|\,,\,|S_n(V)-I_N|=O(\ve_n(W))
\end{equation}
uniformly for $n\ge n^\dia$ and $\det Y_n(V)\ne 0$ for all $V\in\cB^0_\R(V^\dia,r^\dia)$, if
$r^\dia$ is small enough.

\begin{definition}
\label{DefPhi2} Formally introduce the second component of the mapping $\F$ by
\[
\begin{array}{lcl}
\vphantom{\Big|}\F^{(2)} & : & V \mapsto \F^{(2)}(V)=(\F^{(2)}_n(V))_{n=n^\dia}^{+\iy}, \cr
\F^{(2)}_n & = & \left(-i\log U_n\ ;\ 2\pi n\cdot (S_n-I_N)\right)\ :\
\cB^0_\R(V^\dia,r^\dia)\to \C^{N\ts N}_\R\oplus \C^{N\ts N}_\R,
\end{array}
\]
where $\log U_n=(U_n\!-\!I_N)-\frac{1}{2}(U_n\!-\!I_N)^2+\frac{1}{3}(U_n\!-\!I_N)^3-...$
\end{definition}

\noindent Recall that $\wt{A}_{n,j}(V^\dia)=0$ and $\wt{B}_{n,j}(V^\dia)=2\pi^2n^2 P_j^0$ for
all $n\ge
n^\dia$. Thus, 
\[
Y_n(V^\dia)=U_n(V^\dia)=S_n(V^\dia)=I_N\quad \mathrm{and}\quad
\F^{(2)}_n(V^\dia)=(0\,;\,0)\quad\mathrm{for\ all}\ n\ge n^\dia.
\]

\begin{proposition}
\label{PropF2} There exists $r^\dia>0$ such that the mapping
\[
\F^{(2)}:\cB^0_\R(V^\dia,r^\dia)\to\ell^2_\R(\N_{n^\dia} ; \C_\R^{N\ts N}\ts\C_\R^{N\ts N})
\]
is well-defined and real-analytic in $\cB^0_\R(V^\dia,r^\dia)$. Moreover, the Fr\'echet
derivative $d_{V^\dia}\F^{(2)}$ of $\F^{(2)}$ at $V^\dia$ is given by the Fr\'echet
derivatives of its components.
\end{proposition}

\begin{proof}
Due to (\ref{xYUSclose}) and $\sum_{n=1}^{+\iy}|\ve_n(W)|^2=O(\|W\|^2)$, for sufficiently
small $r^\dia>0$ the mapping
\[
\cY:V\mapsto
(Y_n(V)\!-\!I_N)_{n=n^\dia}^{+\infty},\qquad\cB^0_\R(V^\dia,r^\dia)\to\ell^2_\R(\N_{n^\dia};
\C^{N\ts N}),
\]
is well-defined. Recall that $Y_n$ is some simple function of $\wt{A}_{n,j}$ and
$\wt{B}_{n,j}$, $j=1,2,..,N$ (see Definition \ref{aceYUSdef}). Using real-analyticity of the
first two components of the expanded mapping $\P^{(2)}$ (see Definition \ref{DefPsi} and Lemma
\ref{PsiAnalyticity}), we conclude that $\cY$ is real-analytic as a composition of
real-analytic mappings. Since $S_n=(Y_n^*Y_n)^{1/2}$ and $U_n=Y_nS_n^{-1}$, both mappings
\[
\cS:V\mapsto
(S_n(V)\!-\!I_N)_{n=n^\dia}^{+\infty},\qquad\cB^0_\R(V^\dia,r^\dia)\to\ell^2_\R(\N_{n^\dia};
\C^{N\ts N}_\R),
\]
and
\[
\cU:V\mapsto (-i\log U_n(V))_{n=n^\dia}^{+\infty},\ \quad
\cB^0_\R(V^\dia,r^\dia)\to\ell^2_\R(\N_{n^\dia}; \C^{N\ts N}_\R),
\]
are real-analytic too as compositions of $\cY$ with some simple coordinate-wise transforms.

In order to complete the proof it is sufficient to show that $\cS$ actually acts into "better"
space $\ell^2_1$. Note that
\[
Y_nY_n^* = \sum_{j=1}^N c_{n,j}^2\cdot e_{n,j}e_{n,j}^* = \frac{1}{2\pi^2n^2}\sum_{j=1}^N
\wt{B}_{n,j} = \frac{B_n}{2\pi^2n^2}.
\]
Due to Lemma \ref{PsiAnalyticity}, the mapping
\[
\cZ:V\mapsto 2\pi n\cdot (Y_nY_n^*-I_N)_{n=n^\dia}^{+\iy},\qquad \cB^0_\R(V^\dia,r^\dia)\to
\ell^2_\R(\N_{n^\dia},\C^{N\ts N}_\R)
\]
(which is the third component of $\P^{(2)}$) is real-analytic. Using $S_n=
[U_n^{-1}(Y_nY_n^*)U_n]^{1/2}$, we obtain that the mapping
\[
\wt{\cS}: V\mapsto 2\pi n\cdot (S_n(V)\!-\!I_N)_{n=n^\dia}^{+\infty},\qquad
\cB^0_\R(V^\dia,r^\dia)\to\ell^2_\R(\N_{n^\dia}; \C^{N\ts N}_\R),
\]
is real-analytic as a result of some coordinate-wise transforms with $\cZ$ and $\cU$. Note
that $\F^{(2)}=(\cU\,;\,\wt{\cS})$. Since the Fr\'echet derivative $d_{V^\dia}\P$ is given by
the Fr\'echet derivatives of its components, the same holds true for all mappings $\cY$,
$\cS$, $\cU$, $\cZ$ and $\wt{\cS}$.
\end{proof}

\begin{remark}
\label{wtFanaliticity} The mapping $\F=(\F^{(1)}\,;\,\F^{(2)})$ is real-analytic too, since
both $\F^{(1)}$,~$\F^{(2)}$ are real-analytic, and its Fr\'echet derivative is given by the
Fr\'echet derivatives of $\F^{(1)}_\a$,~$\F^{(2)}_n$. Note that each $\F^{(2)}_n$, $n\ge
n^\dia$, contains $2N^2$ real parameters, i.e., exactly "the same amount of information" as,
say, the $n$-th Fourier coefficient $\nh Vn$.
\end{remark}

\subsection{Explicit form of the Fr\'echet derivative $\bf d_{V^\dia}\F$}

\label{SectFrechetDer1}

We denote by
\[
\cP^0: W(x)\mapsto W(x)-\nh W 0
\]
the orthogonal projector in $\cL^2([0,1];\C^{N\ts N}_\R)$ onto $\{W\in \cL^2([0,1];\C^{N\ts
N}_\R):\nh W0 = 0\}$.

Recall that the mapping $\F$ was introduced in Definitions \ref{DefPhi1} and \ref{DefPhi2}.
Due to Remark~\ref{wtFanaliticity}, $(d_{V^\dia}\F)W$ for $W\in\cP^0\cL^2([0,1];\C^{N\ts
N}_\R)$ is given by
\[
\begin{array}{rclclclcl}
(d_{V^\dia}\wt{A}_\a)W, && (d_{V^\dia}C_\a)W, && (d_{V^\dia}E_\a)W && \mathrm{for} &&
\a=1,2,..,\a^\dia\vphantom{|_\big|} \cr \mathrm{and} && (d_{V^\dia}U_n)W, && (d_{V^\dia}S_n)W
&& \mathrm{for} && n\ge n^\dia.
\end{array}
\]
We need some preliminary calculations. Let
\[
\c^\dia_\a=\c(\cdot,\l^\dia_\a,V^\dia),\quad \vp^\dia_\a=\vp(\cdot,\l^\dia_\a,V^\dia)\quad
\mathrm{and\ so\ on}.
\]
Since $V^\dia$ is a {\it diagonal} potential, all these matrix-valued functions are {\it
diagonal}. For short, we will use (a bit careless) notations like
\[
\frac{\c^\dia_\a(t)}{(\c^\dia_\a)'(0)}:=\c^\dia_\a(t)[(\c^\dia_\a)'(0)]^{-1}=
[(\c^\dia_\a)'(0)]^{-1}\c^\dia_\a(t).
\]
Recall that $p_\a^\dia:\C^N\to\cE^\dia_\a$ and $q_\a^\dia:\C^N\to(\cE^\dia_\a)^\perp$ are some
{\it coordinate} projectors. Note that $\Ker[\c^\dia_\a(0)(q_\a^\dia)^*]=\{0\}$,
$\Ker[(\c^\dia_\a)'(0)(p_\a^\dia)^*]=\{0\}$ and $\Ker[\dot\c^\dia_\a(0)(p_\a^\dia)^*]=\{0\}$.
Thus, expressions
\[
[\c^\dia_\a(0)]^{-1}(q_\a^\dia)^*,\quad [(\c^\dia_\a)'(0)]^{-1}(p_\a^\dia)^*\quad
\mathrm{and}\quad [\dot\c^\dia_\a(0)]^{-1}(p_\a^\dia)^*
\]
(and their conjugates) are well-defined.

\begin{proposition}
\label{AGHgrad} For all $\a\ge 1$ and $W\in\cP^0\cL^2([0,1];\C^{N\ts N}_\R)$ the following
hold:
\begin{equation}
\label{wtAgrad} (d_{V^\dia}\wt{A}_\a)W=\, p_\a^\dia
\left[\int_0^1\frac{\c^\dia_\a(t)}{(\c^\dia_\a)'(0)}\,W(t)
\frac{\c^\dia_\a(t)}{(\c^\dia_\a)'(0)}\,dt\right]\!(p_\a^\dia)^*,
\end{equation}
\begin{equation}
\label{Egrad} (d_{V^\dia}E_\a)W\, = - q_\a^\dia
\left[\int_0^1\frac{\c^\dia_\a(t)}{\c^\dia_\a(0)}\,W(t)
\frac{\c^\dia_\a(t)}{(\c^\dia_\a)'(0)}\,dt\right]\!(p_\a^\dia)^*
\end{equation}
and
\begin{equation}
\label{Cgrad} (d_{V^\dia}C_\a)W= p_\a^\dia\lt[
\int_0^1\lt(\frac{\xi^\dia_\a(t)}{\dot\c^\dia_\a(0)}\,W(t)\frac{\c^\dia_\a(t)}{\dot\c^\dia_\a(0)}
+ \frac{\c^\dia_\a(t)}{\dot\c^\dia_\a(0)}\,W(t)\frac{\xi^\dia_\a(t)}{\dot\c^\dia_\a(0)}\rt)dt
\rt]\!(p_\a^\dia)^*,
\end{equation}
where
\begin{equation}
\label{XiDef} \xi^\dia_\a(t)\equiv
\dot\c^\dia_\a(t)-\frac{\ddot\c^\dia_\a(0)}{2\dot\c^\dia_\a(0)}\,\c^\dia_\a(t).
\end{equation}
\end{proposition}

\begin{proof}
It follows from (\ref{ChiW=}) and (\ref{Chi'W=}) that
\[
\begin{array}{l} \displaystyle
(d_{V^\dia}\c(0,\l^\dia_\a))W\!=\!\int_0^1\!\vp^\dia_\a(t)W(t)\c^\dia_\a(t)dt, \quad
(d_{V^\dia}\c'(0,\l^\dia_\a))W\!=-\!\int_0^1\!\vt^\dia_\a(t)W(t)\c^\dia_\a(t)dt
\vphantom{\int_\big|}\cr\displaystyle \mathrm{and}\qquad
(d_{V^\dia}\dot{\c}(0,\l^\dia_\a))W=\int_0^1\left(\dot{\vp}^\dia_\a(t)W(t)\c^\dia_\a(t)+
\vp^\dia_\a(t)W(t)\dot{\c}^\dia_\a(t) \right)dt.
\end{array}
\]
Recall that $\wt{A}_\a = A_\a^{11}- A_\a^{12}(A_\a^{22})^{-1}A_\a^{21}$, where
\[
A_\a(V)=[\c(\c')^{-1}](0,\l_\a^\dia,V),\ \ A_\a^{11}=p_\a^\dia A_\a (p_\a^\dia)^*,\ \
A_\a^{12}=p_\a^\dia A_\a (q_\a^\dia)^*\ \ \mathrm{and\ so\ on.}
\]
Due to $A_\a^{12}(V^\dia)=0$, $A_\a^{21}(V^\dia)=0$ and $p_\a^\dia\c_\a^\dia(0)=0$, one
obtains
\[
\begin{array}{l}
\displaystyle (d_{V^\dia} \wt{A}_\a)W= (d_{V^\dia}A_\a^{11})W =  p_\a^\dia
(d_{V^\dia}\c(0,\l_\a^\dia))\,W\,[(\c_\a^\dia)'(0)]^{-1}(p_\a^\dia)^*
\vphantom{|_\big|}\cr\displaystyle\hphantom{(d_{V^\dia}\wt{A}_\a)W=(d_{V^\dia}A_\a^{11})W} =
p_\a^\dia\lt[\int_0^1\vp^\dia_\a(t)W(t)
\frac{\c^\dia_\a(t)}{(\c^\dia_\a)'(0)}\,dt\rt]\!(p_\a^\dia)^*.
\end{array}
\]
This gives (\ref{wtAgrad}), since $p_\a^\dia\vp_\a^\dia(t)\equiv p_\a^\dia
\c_\a^\dia(t)[(\c_\a^\dia)'(0)]^{-1}$. Next,
\[
\begin{array}{l}
\displaystyle (d_{V^\dia}\wt{B}_\a)W= -\frac{1}{2\pi i}\oint_{|\l-\l_\a^\dia|=d^\dia}
(d_{V^\dia}(\c'\c^{-1})(0,\l))W\,d\l
\vphantom{\int_\big|}\cr\displaystyle\hphantom{(d_{V^\dia}\wt{B}_\a)W} =\frac{1}{2\pi
i}\oint_{|\l-\l_\a^\dia|=d^\dia} \left[-(d_{V^\dia}\c'(0,\l))W+
\frac{(\c^\dia)'(0,\l)}{\c^\dia(0,\l)}\,(d_{V^\dia}\c(0,\l))W\right]\!\frac{d\l}{\c^\dia(0,\l)}\,.\!
\end{array}
\]

Note that the diagonal matrix-valued function $[\c^\dia(0,\l)]^{-1}$ has the unique pole
(at~$\l_\a^\dia$) inside of the contour of integration and
\[
\frac{I_N}{\c^\dia(0,\l)}= P_\a^\dia\lt[\frac{I_N}{\dot\c^\dia_\a(0)(\l\!-\!\l_\a^\dia)}-
\frac{\ddot\c^\dia_\a(0)}{2[\dot\c^\dia_\a(0)]^2}\rt]P_\a^\dia\, +\, Q_\a^\dia
\frac{I_N}{\c^\dia_\a(0)}\,Q_\a^\dia + O(\l\!-\!\l_\a^\dia)\ \ \mathrm{as}\ \l\to\l_\a,
\]
where $Q_\a^\dia=(q_\a^\dia)^*q_\a^\dia=I_N-P_\a^\dia$. Recall that $E_\a=\wt{B}^{21}_\a
[\wt{B}^{11}_\a]^{-1}$ and $\wt{B}_\a^{21}(V^\dia)=0$.
Thus,
\[
(d_{V^\dia}E_\a)W= (d_{V^\dia}\wt{B}_\a^{21})W \cdot [\wt{B}_\a^{11}(V^\dia)]^{-1} = -
 (d_{V^\dia}\wt{B}_\a^{21})W \cdot p_\a^\dia
\frac{\dot\c_\a^\dia(0)}{(\c_\a^\dia)'(0)}\, (p_\a^\dia)^*
\]
and
\[
(d_{V^\dia}\wt{B}_\a^{21})W= 
q_\a^\dia\lt[ \int_0^1
\lt(\vt^\dia_\a(t)+\frac{(\c^\dia_\a)'(0)}{\c^\dia_\a(0)}\,\vp^\dia_\a(t)\rt)W(t)
\c^\dia_\a(t)\frac{I_N}{\dot\c^\dia_\a(0)}\,dt\rt](p_\a^\dia)^*.
\]
Using $\c^\dia_\a(0)\vt^\dia_\a(t)+(\c^\dia_\a)'(0)\vp^\dia_\a(t)\equiv \c^\dia_\a(t)$, one
obtains (\ref{Egrad}).

\smallskip

Furthermore, $C_\a(V)=\wt{B}_\a^{11}(V)=p^\dia_\a \wt{B}_\a(V) (p^\dia_\a)^*$. In contrast to
$(d_{V^\dia}\wt{B}^{21}_\a)W$, we do not have cancellations of the singularities by the
projectors, so one should find the residue at the second order pole $\l_\a^\dia$.
Straightforward calculations give
\[
(d_{V^\dia}C_\a)W \ = \ \res_{\l=\l_\a} p_\a^\dia\lt[-(d_{V^\dia}\c'(0,\l))W+
\frac{(\c^\dia)'(0,\l)}{\c^\dia(0,\l)}\,(d_{V^\dia}\c(0,\l))W\rt]
\frac{I_N}{\c^\dia(0,\l)}\,(p^\dia_\a)^*
\]
\[
\begin{array}{l}
\displaystyle  = p_\a^\dia\lt[\int_0^1 \lt(\lt[\,\vt^\dia_\a(t)W(t)\c_\a^\dia(t)+
\lt(\frac{(\dot\c^\dia_\a)'(0)}{\dot\c^\dia_\a(0)}\,-\frac{(\c^\dia_\a)'(0)\ddot\c^\dia_\a(0)}
{2[\dot\c^\dia_\a(0)]^2}\rt)\vp^\dia_\a(t)W(t)\c^\dia_\a(t)
\vphantom{\int_\big|}\cr\displaystyle\hphantom{=p_\a^\dia\lt[\int_0^1\lt(\lt[\,\vt^\dia_\a(t)W(t)\c_\a^\dia(t)}
+\frac{(\c^\dia_\a)'(0)}{\dot\c^\dia_\a(0)}
\left(\vphantom{\Big|}\dot{\vp}^\dia_\a(t)W(t)\c^\dia_\a(t)+
\vp^\dia_\a(t)W(t)\dot{\c}^\dia_\a(t)\right)\rt]\frac{I_N}{\dot\c^\dia_\a(0)}
\vphantom{\int_\big|^\big|}\cr\displaystyle\hphantom{=p_\a^\dia\lt[\int_0^1\lt(}
-\,\frac{I_N}{\dot\c^\dia_\a(0)} (\c^\dia_\a)'(0)\vp^\dia_\a
(t)W(t)\frac{\c^\dia_\a(t)\ddot\c^\dia_\a(0)}{2[\dot\c^\dia_\a(0)]^2}\rt)dt\rt]\!(p_\a^\dia)^*.
\vphantom{\int_\big|}
\end{array}
\]
Using the identities
\[
\vt^\dia_\a(t) + \frac{(\dot\c^\dia_\a)'(0)}{\dot\c^\dia_\a(0)}\,\vp^\dia_\a(t)+
\frac{(\c^\dia_\a)'(0)}{\dot\c^\dia_\a(0)}\,\dot{\vp}^\dia_\a(t) \equiv
\frac{\dot{\c}^\dia_\a(t)}{\dot\c^\dia_\a(0)}
\]
and  $p_\a^\dia (\c^\dia_\a)'(0)\vp^\dia_\a(t) \equiv p_\a^\dia \c^\dia_\a(t)$, one obtains
(\ref{Cgrad}).
\end{proof}

Introduce the functions
\[
\c_\a^{\dia,j}(t)\equiv [\c_\a^\dia(t)]_{jj} \equiv \c(t,\l_\a^\dia,v^\dia_{jj})\quad
\mathrm{and}\quad \x_\a^{\dia,j}(t)\equiv [\x_\a^\dia(t)]_{jj} \equiv
\x(t,\l_\a^\dia,v^\dia_{jj}),
\]
where $\x^\dia_\a$ is given by (\ref{XiDef}).

\begin{corollary}
\label{ACEgradCoord} Let $\a\ge 1$ and $I(\a)=\{s:\l_\a^\dia\in\s(v_{ss}^\dia)\}$ (by
definition, the set $I(\a)$ consists of $k_\a^\dia$ indices). Then, for all
$W\in\cP^0\cL^2([0,1];\C^{N\ts N}_\R)$,
\[
[(d_{V^\dia}\wt{A}_\a)W]_{jk}=\langle W_{jk}\,,\,u_\a^{(jk)}\rangle,\qquad
[(d_{V^\dia}C_\a)W]_{jk}=\langle W_{jk}\,,\,\wt{u}_\a^{(jk)}\rangle,\qquad j,k\in I(\a),
\]
where for all $\l_\a^\dia\in\s(v_{jj}^\dia)\cap\s(v_{kk}^\dia)$ the functions $u_\a^{(jk)}$
and $\wt{u}_\a^{(jk)}$ are given by
\begin{equation}
\label{UaJKdef1}
\begin{array}{l}
u_\a^{(jk)}(t)\equiv [(\c^{\dia,j}_\a)'(0)(\c^{\dia,k}_\a)'(0)]^{-1}\cdot
\c^{\dia,j}_\a(t)\c^{\dia,k}_\a(t),\vphantom{|_\big|}\cr \wt{u}_\a^{(jk)}(t) \equiv
[\dot\c^{\dia,j}_\a(0)\dot\c^{\dia,k}_\a(0)]^{-1}\cdot
\left(\xi^{\dia,j}_\a(t)\c^{\dia,k}_\a(t)+ \c^{\dia,j}_\a(t)\xi^{\dia,k}_\a(t)\right).
\end{array}
\end{equation}
Furthermore,
\[
[(d_{V^\dia}E_\a)W]_{jk}= \langle W_{jk}\,,\,{u}_\a^{(jk)}\rangle,\qquad j\notin I(\a),\ k\in
I(\a),
\]
where for all $\l_\a^\dia\in\s(v_{kk}^\dia)\setminus \s(v_{jj}^\dia)$ the function
$u_\a^{(jk)}$ is given by
\begin{equation}
\label{UaJKdef2} u_\a^{(jk)}(t)\equiv \,-[\c^{\dia,j}_\a(0)(\c^{\dia,k}_\a)'(0)]^{-1}\cdot
\c^{\dia,j}_\a(t)\c^{\dia,k}_\a(t).
\end{equation}
\end{corollary}
\begin{proof}
Since $\c^\dia_\a$, $\x^\dia_\a$ are diagonal matrices, this is exactly the result of
Proposition~\ref{AGHgrad} rewritten in the coordinate form.
\end{proof}

\begin{proposition}
\label{PropUSgrad} Let $n\ge n^\dia$ and $j,k=1,2,..,N$ be such that $j\ne k$. Then for all
$W\in\cP^0\cL^2([0,1];\C^{N\ts N}_\R)$ the following identities hold:
\begin{equation}
\label{YjjGrad} [(d_{V^\dia}Y_n)W]_{jj}= (4\pi^2n^2)^{-1}\langle
W_{jj}\,,\,\wt{u}_{n,j}^{(jj)}\rangle + i\cdot 2\pi^2n^2\langle
W_{jj}\,,\,u_{n,j}^{(jj)}\rangle,
\end{equation}
where the functions $u_{n,j}^{(jj)}$ and $\wt{u}_{n,j}^{(jj)}$ are given by (\ref{UaJKdef1}),
and
\begin{equation}
\label{YjkGrad} [(d_{V^\dia}Y_n)W]_{jk}= \langle W_{jk}\,,\,u_{n,k}^{(jk)}\rangle,
\end{equation}
where the functions $u_{n,k}^{(jk)}$ are given by (\ref{UaJKdef2}). Furthermore,
\begin{equation}
\label{SUgrad}
\begin{array}{l}
(d_{V^\dia}S_n)W = \frac{1}{2}\left((d_{V^\dia}Y_n)W+[(d_{V^\dia}Y_n)W]^*\right),
\vphantom{|_\big|}\cr (d_{V^\dia}U_n)W =
\frac{1}{2}\left((d_{V^\dia}Y_n)W-[(d_{V^\dia}Y_n)W]^*\right).
\end{array}
\end{equation}
\end{proposition}

\begin{proof}
By definition of $Y_n$,
\[
[(d_{V^\dia}Y_n)W]_{jk}= \left\langle(d_{V^\dia}[\exp(ia_{n,k})\cdot c_{n,k}\cdot
e_{n,k}])W\,,\,e_j^0\right\rangle.
\]
Recall that $a_{n,k}(V^\dia)=0$, $c_{n,k}(V^\dia)=1$, $e_{n,k}(V)=e_k^0+E_{n,k}(V)e_k^0$ and
$E_{n,k}(V^\dia)=0$. Thus,
\[
[(d_{V^\dia}Y_n)W]_{jj}=(d_{V^\dia}c_{n,j})W + i\cdot (d_{V^\dia}a_{n,j})W=
\frac{(d_{V^\dia}C_{n,j})W}{4\pi^2n^2} +i\cdot 2\pi^2n^2 (d_{V^\dia}\wt{A}_{n,j})W
\]
and
\[
[(d_{V^\dia}Y_n)W]_{jk}= [(d_{V^\dia}E_{n,k})W]_{j}\,.
\]
Due to Corollary \ref{ACEgradCoord}, one obtains (\ref{YjjGrad}) and (\ref{YjkGrad}). Recall
that $S_n=(Y_n^*Y_n)^{1/2}$, $U_n=Y_nS_n^{-1}$ and $Y_n(V^\dia)=U_n(V^\dia)=S_n(V^\dia)=I_N$.
This immediately gives (\ref{SUgrad}).
\end{proof}

\subsection{Invertibility of the Fr\'echet derivative $\bf d_{V^\dia}\F$}

\label{SectFrechetDer2} Due to Remark \ref{wtFanaliticity},
\[
\begin{array}{ll}
(d_{V^\dia}\F^{(1)}_\a)W =
((d_{V^\dia}\wt{A}_\a)W\,;\,(d_{V^\dia}C_\a)W\,;\,(d_{V^\dia}E_\a)W), & \a=1,2,..,\a^\dia,
\vphantom{|_\big|}\cr (d_{V^\dia}\F^{(2)}_n)W = (-i(d_{V^\dia}U_n)W\,;\,2\pi
n(d_{V^\dia}S_n)W), & n=n^\dia,n^\dia\!+\!1,...
\end{array}
\]
Recall that $W_{kj}=\ol{W_{jk}}$ for all $1\!\le\!k\!\le\!j\!\le\!N$. It immediately follows
from Corollary~\ref{ACEgradCoord} and Proposition \ref{PropUSgrad} that the entries of the
components of $(d_{V^\dia}\F)W$ are { \it
\begin{enumerate}

\smallskip

\item  for all $j=1,2,..,N$
(diagonal entries of {\rm (a)} $\wt{A}_\a$, $C_\a$ and {\rm (b)} $U_n$, $S_n$):

\smallskip

\begin{enumerate}
\item $\langle W_{jj}\,,\,u_\a^{(jj)}\rangle$,\quad $\langle
W_{jj}\,,\,\wt{u}_\a^{(jj)}\rangle$,\quad where $\a\le \a^\dia$ are such that\ \
$\l_\a^\dia\in\s(v_{jj}^\dia)$;

\smallskip

\item $2\pi^2n^2\cdot \langle W_{jj}\,,\,u_{n,j}^{(jj)}\rangle$,\quad $(2\pi n)^{-1}\cdot
\langle W_{jj}\,,\,\wt{u}_{n,j}^{(jj)}\rangle$,\quad for all $n\ge n^\dia$;
\end{enumerate}

\smallskip

\item   for all $1\!\le\!k\!<\!j\!\le\! N$
(non-diagonal entries of {\rm (a)} $\wt{A}_\a$, $C_\a$; {\rm (b)} $E_\a$; {\rm (c)} $U_n$,
$S_n$):

\smallskip

\begin{enumerate}
\item $\langle W_{jk}\,,\,u_\a^{(jk)}\rangle$,\quad  $\langle
W_{jk}\,,\,\wt{u}_\a^{(jk)}\rangle$\quad and their complex-conjugates\\ $\langle
\ol{W_{jk}}\,,\,u_\a^{(kj)}\rangle$,\quad $\langle \ol{W_{jk}}\,,\,\wt{u}_\a^{(kj)}\rangle$, \
where $\a\le \a^\dia$: $\l_\a^\dia\in\s(v_{jj}^\dia)\cap\s(v_{kk}^\dia)$;

\smallskip

\item $\langle W_{jk}\,,\,u_\a^{(jk)}\rangle$, where $\a\le \a^\dia$ are such that
$\l_\a^\dia\in\s(v_{kk}^\dia)\setminus\s(v_{jj}^\dia)$; \\
$\langle \ol{W_{jk}}\,,\,u_\a^{(kj)}\rangle$, where $\a\le \a^\dia$ are such that
$\l_\a^\dia\in\s(v_{jj}^\dia)\setminus\s(v_{kk}^\dia)$;

\smallskip

\item $\frac{1}{2i}\cdot\langle W_{jk}\,,\,[u_{n,k}^{(jk)}\!-\!u_{n,j}^{(kj)}]\rangle$,\quad
$\pi n\cdot \langle W_{jk}\,,\,[u_{n,k}^{(jk)}\!+\!u_{n,j}^{(kj)}]\rangle$\quad and their conjugates \\
$\frac{1}{2i}\cdot\langle \ol{W_{jk}}\,,\,[u_{n,j}^{(kj)}\!-\!u_{n,k}^{(jk)}]\rangle$,\quad
$\pi n\cdot \langle \ol{W_{jk}}\,,\,[u_{n,j}^{(kj)}\!+\!u_{n,k}^{(jk)}]\rangle$,\quad  for all
$n\ge n^\dia$.
\end{enumerate}

\smallskip

\end{enumerate} }

\smallskip

\noindent Note that $u_\a^{(jk)}=u_\a^{(kj)}$ and $\wt{u}_\a^{(jk)}=\wt{u}_\a^{(kj)}$, if
$\l_\a\in\s(v_{jj}^\dia)\cap\s(v_{kk}^\dia)$.

\begin{definition}
\label{cUdef} For each $1\!\le\!k\!\le\!j\!\le\!N$ we introduce the collection of real scalar
functions
\[
\quad \cU^{(jj)} =\left\{\vphantom{\big|^|_|}u_\a^{(jj)}\,,\,\wt{u}_\a^{(jj)},\
\a\!\le\!\a^\dia:
\l_\a^\dia\!\in\!\s(v_{jj}^\dia)\right\}
\cup \left\{\vphantom{\big|^|_|}2\pi^2n^2u_{n,j}^{(jj)}\,,\,(2\pi
n)^{-1}\,\wt{u}_{n,j}^{(jj)},\ n\!\ge\!n^\dia\right\}
\]
\[
\begin{array}{ccl}
\vphantom{\big|_\big|}\cU^{(jk)} & \!\!\!=\!\!\! & \left\{u_\a^{(jk)},\wt{u}_\a^{(jk)},\
\a\!\le\!\a^\dia:\l_\a^\dia\!\in\!\s(v_{jj}^\dia)\cap\s(v_{kk}^\dia)\right\}
\cr\vphantom{\big|^\big|_\big|} &  \!\!\!\cup\!\!\! & \left\{\vphantom{\big|^|_|}u_\a^{(jk)},\
\a\!\le\!\a^\dia:\l_\a^\dia\!\in\!\s(v_{kk}^\dia)\setminus\s(v_{jj}^\dia)\right\}
\,\cup \left\{\vphantom{\big|^|_|}u_\a^{(kj)},\ \a\!\le\!\a^\dia:
\l_\a^\dia\!\in\!\s(v_{jj}^\dia)\setminus\s(v_{kk}^\dia)\right\}\cr \vphantom{\big|^\big|} &
\!\!\!\cup\!\!\! &
\left\{\vphantom{\big|^|_|}{\frac{1}{2}}[u_{n,k}^{(jk)}-u_{n,j}^{(kj)}]\,,\,\pi
n[u_{n,k}^{(jk)}+u_{n,j}^{(kj)}]\,,\ n\!\ge\!n^\dia\right\},
\end{array}
\]
where the functions $u_\a^{(jk)}$ and $\wt{u}_\a^{(jk)}$ are given by (\ref{UaJKdef1}) and
(\ref{UaJKdef2}). Note that each collection $\cU^{(jk)}$ contains exactly $2(n^\dia\!-\!1)$
functions with "small" indices $\a\le\a^\dia$.
\end{definition}

\begin{remark}
\label{RemarkRiesz} Due to the arguments given above, in order to prove that
$[d_{V^\dia}\F]^{-1}$ is bounded, it is sufficient to prove that each $\cP^0\cU^{(jk)}$ is
a~Riesz basis of $\cP^0 \cL^2(0,1)$.
\end{remark}

\begin{lemma}
\label{BiortLemma} For each $1\!\le\!k\!\le\!j\!\le\!N$ there exists some collection of
functions \mbox{$\cV^{(jk)}\ss \cP^0 \cL^2(0,1)$} which is biorthogonal to $\cU^{(jk)}$ (and,
therefore, to $\cP^0\cU^{(jk)}$).
\end{lemma}

\begin{proof}
Taking into account definitions (\ref{UaJKdef1}), (\ref{UaJKdef2}) and (\ref{XiDef}), it is
sufficient to construct some collection $\wt{\cV}^{(jk)}\ss \cP^0 \cL^2(0,1)$ which is
biorthogonal to $\cP^0\wt{\cU}^{(jk)}$, where
\[
\begin{array}{rcl}
\vphantom{|_\big|}\wt{\cU}^{(jk)} & = &\left\{\c^{\dia,j}_\a\c^{\dia,k}_\a,\ \mathrm{for\
all}\ \l_\a^\dia\in\s(v_{jj}^\dia)\cup\s(v_{kk}^\dia)\right\}\cr \vphantom{|^\big|}&\cup &
\left\{\dot\c^{\dia,j}_\a\c^{\dia,k}_\a+\c^{\dia,j}_\a\dot\c^{\dia,k}_\a, \ \mathrm{for\ all}\
\l_\a^\dia\in \s(v_{jj}^\dia)\cap\s(v_{kk}^\dia)\right\},
\end{array}
\]
since $\wt{\cU}^{(jk)}$ and ${\cU}^{(jk)}$ are related by some simple linear transformations
(namely, multiplications by fixed constants, $(\c,\x\!=\!\dot{\c}\!+\!c\c)\leftrightarrow
(\c,\dot{\c})$ and $(u_1,u_2)\leftrightarrow (u_1\!+\!u_2,u_1\!-\!u_2)$). Note that we
consider both cases $k=j$ and $k<j$ simultaneously. Let
\[
\begin{array}{rcl}
\vphantom{|_\big|}\wt{\cV}^{(jk)} & = &\left\{[\vp^{\dia,j}_\b\vp^{\dia,k}_\b]',\ \mathrm{for\
all}\ \l_\b^\dia\in\s(v_{jj}^\dia)\cup\s(v_{kk}^\dia)\right\}\cr \vphantom{|^\big|}&\cup &
\left\{[\dot\vp^{\dia,j}_\b\vp^{\dia,k}_\b+\vp^{\dia,j}_\b\dot\vp^{\dia,k}_\b]', \
\mathrm{for\ all}\ \l_\b^\dia\in \s(v_{jj}^\dia)\cap\s(v_{kk}^\dia)\right\},
\end{array}
\]
By definition, $\wt{\cV}^{(jk)}\ss \cP^0 \cL^2(0,1)$. Let $\l_\a\ne\l_\b$ and
$\{\c,\vp\}=\c\vp'-\c'\vp$. The standard trick (e.g., see \cite{PT} pp. 44--45 for the similar
calculation in the scalar case ) shows
\[
\left\langle \c^{\dia,j}_\a\c^{\dia,k}_\a,[\vp^{\dia,j}_\b\vp^{\dia,k}_\b]'\right\rangle =
\frac{1}{2}\int_0^1 \left[(\c^{\dia,j}_\a\c^{\dia,k}_\a)(\vp^{\dia,j}_\b\vp^{\dia,k}_\b)'-
(\c^{\dia,j}_\a\c^{\dia,k}_\a)'(\vp^{\dia,j}_\b\vp^{\dia,k}_\b)\right]\!(t)dt
\]
\begin{equation}
\label{xCCFF} =
\frac{1}{2}\int_0^1\left[\{\c^{\dia,j}_\a,\vp^{\dia,j}_\b\}(\c^{\dia,k}_\a\vp^{\dia,k}_\b) +
(\c^{\dia,j}_\a\vp^{\dia,j}_\b)\{\c^{\dia,k}_\a,\vp^{\dia,k}_\b\}\right]\!(t)dt
\end{equation}
\[ =
\frac{\{\c^{\dia,j}_\a,\vp^{\dia,j}_\b\}\{\c^{\dia,k}_\a,\vp^{\dia,k}_\b\}\big|_0^1}{2(\l_\a\!-\!\l_\b)}
= \frac{[\vp^{\dia,j}_\b\vp^{\dia,k}_\b](1)-
[\c^{\dia,j}_\a\c^{\dia,k}_\a](0)}{2(\l_\a\!-\!\l_\b)}.\phantom{[}
\]

If both $\l_\a,\l_\b\in \s(v_{jj}^\dia)\cup\s(v_{kk}^\dia)$, then
$\vp^{\dia,j}_\b(1)\vp^{\dia,k}_\b(1)=\c^{\dia,j}_\a(0)\c^{\dia,k}_\a(0)=0$. Hence,
\[
\left\langle \c^{\dia,j}_\a\c^{\dia,k}_\a,[\vp^{\dia,j}_\b\vp^{\dia,k}_\b]'\right\rangle =0
\]
Moreover, if $\l_\a\in\s(v_{jj}^\dia)\cap\s(v_{kk}^\dia)$ (the case
$\l_\b\in\s(v_{jj}^\dia)\cap\s(v_{kk}^\dia)$ is similar), then the right-hand side in
(\ref{xCCFF}), as a function of $\l_\a$, has a double zero, so we can differentiate this
identity (with respect to $\l_\a$) and obtain
\[
\left\langle
\dot\c^{\dia,j}_\a\c^{\dia,k}_\a\!+\!\c^{\dia,j}_\a\dot\c^{\dia,k}_\a,[\vp^{\dia,j}_\b\vp^{\dia,k}_\b]'\right\rangle
=0.
\]
Also, if both $\l_\a,\l_\b\in\s(v_{jj}^\dia)\cap\s(v_{kk}^\dia)$, then
\[
\left\langle \dot\c^{\dia,j}_\a\c^{\dia,k}_\a\!+\!\c^{\dia,j}_\a\dot\c^{\dia,k}_\a,
[\dot\vp^{\dia,j}_\b\vp^{\dia,k}_\b\!+\!\vp^{\dia,j}_\b\dot\vp^{\dia,k}_\b]'\right\rangle =0.
\]

Let $\l_\a\!=\!\l_\b\in\s(v_{jj}^\dia)\setminus\s(v_{kk}^\dia)$ (or
$\l_\a\!=\!\l_\b\in\s(v_{kk}^\dia)\setminus\s(v_{jj}^\dia)$). Then
$\{\c^{\dia,j}_\a,\vp^{\dia,j}_\a\}=0$, $\{\c^{\dia,k}_\a,\vp^{\dia,k}_\a\}\ne 0$ and
\[
\left\langle \c^{\dia,j}_\a\c^{\dia,k}_\a\,, [\vp^{\dia,j}_\a\vp^{\dia,k}_\a]'\right\rangle =
\frac{\{\c^{\dia,k}_\a,\vp^{\dia,k}_\a\}}{2}\int_0^1\c^{\dia,j}_\a(t)\vp^{\dia,j}_\a(t)dt\ne
0.
\]
Let $\l_\a\!=\!\l_\b\in \s(v_{jj}^\dia)\cap\s(v_{kk}^\dia)$. Then
$\{\c^{\dia,j}_\a,\vp^{\dia,j}_\a\}=\{\c^{\dia,k}_\a,\vp^{\dia,k}_\a\}=0$ and
\[
\left\langle \c^{\dia,j}_\a\c^{\dia,k}_\a\,,[\vp^{\dia,j}_\a\vp^{\dia,k}_\a]'\right\rangle =
0.
\]
Using (\ref{xCCFF}) for $\l_\b\to\l_\a$, one gets
\[
\left\langle \dot\c^{\dia,j}_\a\c^{\dia,k}_\a\!+\!\c^{\dia,j}_\a\dot\c^{\dia,k}_\a\,,\,
[\vp^{\dia,j}_\a\vp^{\dia,k}_\a]'\right\rangle = \lim_{\l_\b\to \l_\a}
\frac{\vp^{\dia,j}_\b(1)\vp^{\dia,k}_\b(1)}{2(\l_\a\!-\!\l_\b)^2}=
\frac{[\dot\vp^{\dia,j}_\a\dot\vp^{\dia,k}_\a](1)}{2}\ne 0.
\]
Similarly,
\[
\left\langle \c^{\dia,j}_\a\c^{\dia,k}_\a\,,[\dot\vp^{\dia,j}_\a\vp^{\dia,k}_\a +
\vp^{\dia,j}_\a\dot\vp^{\dia,k}_\a]'\right\rangle = -
\frac{[\dot\c^{\dia,j}_\a\dot\c^{\dia,k}_\a](0)}{2}\ne 0.
\]
Finally, one needs to correct $\wt{\cV}^{(jk)}$ slightly, replacing the functions
$[\dot\vp^{\dia,j}_\a\vp^{\dia,k}_\a + \vp^{\dia,j}_\a\dot\vp^{\dia,k}_\a]'$ for all
$\l_\a\in\s(v_{jj}^\dia)\cap\s(v_{kk}^\dia)$) by
\[
[\dot\vp^{\dia,j}_\a\vp^{\dia,k}_\a + \vp^{\dia,j}_\a\dot\vp^{\dia,k}_\a]'+
c_\a[\vp^{\dia,j}_\a\vp^{\dia,k}_\a]'
\]
with appropriate constants $c_\a$, in order to guarantee
\[
\left\langle
\dot\c^{\dia,j}_\a\c^{\dia,k}_\a\!+\!\c^{\dia,j}_\a\dot\c^{\dia,k}_\a\,,\,[\dot\vp^{\dia,j}_\a\vp^{\dia,k}_\a
+ \vp^{\dia,j}_\a\dot\vp^{\dia,k}_\a]'+ c_\a[\vp^{\dia,j}_\a\vp^{\dia,k}_\a]'\right\rangle = 0
\]
After these corrections, $\wt{\cV}^{(jk)}$ becomes biorthogonal to $\wt{\cU}^{(jk)}$.
\end{proof}

\begin{proposition}
\label{PropRiesz} $\cP^0\cU^{(jk)}$ is a Riesz basis of $\cP^0\cL^2(0,1)$ for all
$1\!\le\!k\!\le\!j\!\le N$.
\end{proposition}

\begin{proof} Since $\cP^0\cU^{(jk)}$ admits the biorthogonal system, it is sufficient to
check that elements of $\cP^0\cU^{(jk)}$ are asymptotically close (say, in $\ell^2$--sense) to
some unperturbed Riesz basis (note that these functions are in one-to-one correspondence with
eigenvalues of $v_{jj}^\dia$ and $v_{kk}^\dia$, and we have two functions in $\cU^{(jk)}$ for
common eigenvalues). Those $u\in\cU^{(jk)}$ that correspond to first eigenvalues
$\l_{n,j}^\dia$, $\l_{n,k}^\dia$, $n<n^\dia$, do not affect the asymptotical behavior, so it
is sufficient to consider $n\ge n^0$.

We need some simple asymptotics. Let $\l=\pi^2n^2\!+\!\m$, $\m=O(1)$, and $v\in \cL^2(0,1)$ be
some (scalar) potential. Then
\[
\c(t,\l,v)=\frac{\sin\pi n (1\!-\!t)}{\pi n} + O\lt(\frac{1}{n^2}\rt),\qquad \dot\c(t,\l,v)=
\frac{(1\!-\!t)\cos \pi n(1\!-\!t)}{2\pi^2n^2} + O\lt(\frac{1}{n^3}\rt),
\]
\[
\c'(0,\l,v)=(-1)^{n-1}+O\lt(\frac{1}{n}\rt),\quad\dot\c(0,\l,v)= \frac{(-1)^n}{2\pi^2n^2} +
O\lt(\frac{1}{n^3}\rt),\quad \ddot\c(0,\l,v)=O\lt(\frac{1}{n^4}\rt)
\]
as $n\to\iy$. In particular,
\[
\x(t,\l,v)= \dot\c(t,\l,v)- \frac{\ddot\c(0,\l,v)}{2\dot \c(0,\l,v)}\,\c(t,\l,v)=
\frac{(1\!-\!t)\cos \pi n(1\!-\!t)}{2\pi^2n^2} + O\lt(\frac{1}{n^3}\rt).
\]

If $k=j$, one obtains
\[
\cP^0\lt[2\pi^2n^2\cdot u_{n,j}^{(jj)}\rt] =
\cP^0\lt[\frac{2\pi^2n^2[\c^{\dia,j}_{n,j}(t)]^2}{[(\c^{\dia,j}_{n,j})'(0)]^2}\rt] = - \cos
2\pi n t + O\lt(\frac{1}{n}\rt)
\]
and
\[
\cP^0\lt[(2\pi n)^{-1}\!\cdot \wt{u}_{n,j}^{(jj)}\rt] = \cP^0\lt[
\frac{\xi^{\dia,j}_{n,j}(t)\c^{\dia,j}_{n,j}(t)}{\pi n [\dot\c^{\dia,j}_{n,j}(0)]^2}\rt] =
-\cP^0\left[\vphantom{\big|}(1\!-\!t)\sin 2\pi nt\right] + O\lt(\frac{1}{n}\rt).
\]
It's easy to see that the collection
\begin{equation}
\label{RieszUnpert} \cR=\left\{\vphantom{\Big|}\cos 2\pi n t\ ,\ \cP^0[(1\!-\!t)\sin 2\pi
nt],\ n\ge 1\right\}
\end{equation}
is a Riesz basis of $\cP^0 \cL^2(0,1)$. Indeed, all functions $(\frac{1}{2}\!-\!t)\sin 2\pi
nt$, $n\ge 1$, are linear combinations of $\cos 2\pi mt$, $m\ge 1$, since they are symmetric
with respect to $\frac{1}{2}$. Hence,
\[
\left(\begin{array}{c} \langle f,\cos 2\pi nt\rangle_{n=1}^{+\iy}\vphantom{|_\big|} \cr
\langle f,\cP^0[(1\!-\!t)\sin 2\pi nt]\rangle_{n=1}^{+\iy}
\end{array}\right) = \left(\begin{array}{cc} I\vphantom{|_\big|} &\ 0 \cr \cA &\ \frac{1}{2}I\end{array}\right)
\left(\begin{array}{c} \langle f,\cos 2\pi nt\rangle_{n=1}^{+\iy} \vphantom{|_\big|}\cr
\langle f,\sin 2\pi nt\rangle_{n=1}^{+\iy} \end{array}\right).
\]
and the linear operator $
\langle f, \cos 2\pi nt\rangle_{n=1}^{+\infty}\mapsto \langle f,
\cP^0[({\textstyle\frac{1}{2}}\!-\!t)\sin 2\pi nt]\rangle_{n=1}^{+\infty}$, 
$f\in\cL^2(0,1)$, is bounded in $\ell^2$, since the operator $f\mapsto (\frac{1}{2}\!-\!t)f$
is bounded in $\cL^2(0,1)$.

Thus, \mbox{$\cR$ is} a Riesz basis of $\cP^0\cL^2(0,1)$ and $\cP^0\cU^{(jj)}$ is
$\ell^2$--close to $\cR$ (note that in both $\cP^0\cU^{(jj)}$ and $\cR$ there are exactly
$2(n^\dia\!-\!1)$ functions with $n<n^\dia$). Due to Lemma~\ref{BiortLemma}, the elements of
$\cP^0\cU^{(jj)}$ are linearly independent. Therefore, $\cP^0\cU^{(jj)}$ is a Riesz basis of
$\cP^0\cL^2(0,1)$ by the Fredholm Alternative (see, e.g., \cite{PT} p.~163).

\smallskip

Let $k<j$ and $n\ge n^\dia$. Due to
$[(\c^{\dia,j}_{n,k})'(\dot\c^{\dia,j}_{n,k})^{-1}](0)\!=\! -(g_{n,k}^\dia)^{-1} \!=\!
-2\pi^2n^2$, one has
\[
u_{n,k}^{(jk)}(t)=
-\frac{\c^{\dia,j}_{n,k}(t)\c^{\dia,k}_{n,k}(t)}{\c^{\dia,j}_{n,k}(0)(\c^{\dia,k}_{n,k})'(0)}
=
\frac{\c^{\dia,j}_{n,k}(t)\c^{\dia,k}_{n,k}(t)}{2\pi^2n^2\c^{\dia,j}_{n,k}(0)\dot\c^{\dia,k}_{n,k}(0)}\,.
\]
Note that
\[
\c^{\dia,j}_{n,k}(t)=\c^{\dia,j}_{n,j}(t)+
(\l^\dia_{n,k}\!-\!\l^\dia_{n,j})\dot\c^{\dia,j}_{n,j}(t) + O(n^{-3})
\]
and
\[
\c^{\dia,j}_{n,k}(0) = (\l_{n,k}^\dia\!-\!\l_{n,j}^\dia)\cdot
\dot\c^{\dia,j}_{n,j}(0)+O(n^{-4}),
\]
since $\c^{\dia,j}_{n,j}(0)=0$ and $\ddot\c^{\dia,j}_{n,j}(0)=O(n^{-4})$. Therefore,
\[
u_{n,k}^{(jk)}(t)= \frac{1}{\l_{n,k}^\dia\!-\!\l_{n,j}^\dia}\cdot
\frac{\c^{\dia,j}_{n,j}(t)\c^{\dia,k}_{n,k}(t)}
{2\pi^2n^2\dot\c^{\dia,j}_{n,j}(0)\dot\c^{\dia,k}_{n,k}(0)}
+\frac{\dot\c^{\dia,j}_{n,j}(t)\c^{\dia,k}_{n,k}(t)}
{2\pi^2n^2\dot\c^{\dia,j}_{n,j}(0)\dot\c^{\dia,k}_{n,k}(0)}+O\lt(\frac{1}{n^2}\rt).
\]
Thus,
\[
\cP^0\lt[{\textstyle\frac{1}{2}}[u_{n,k}^{(jk)}-u_{n,j}^{(kj)}]\rt]= -\frac{\cos 2\pi
nt}{\l_{n,k}^\dia\!-\!\l_{n,j}^\dia}+ O\lt(\frac{1}{n}\rt)
\]
and, since the first term of $u_{n,k}^{(jk)}(t)$ is antisymmetric with respect to $j$ and $k$,
\[
\cP^0\lt[\pi n\cdot [u_{n,k}^{(jk)}+u_{n,j}^{(kj)}]\rt]=
-\cP^0\left[\vphantom{\big|}(1\!-\!t)\sin 2\pi nt\right] +O\lt(\frac{1}{n}\rt).
\]
As above, we see that $\cP^0\cU^{(jk)}$ (up to some uniformly bounded multiplicative
constants) is $\ell^2$--close to the Riesz basis $\cR$ given by (\ref{RieszUnpert}). So,
$\cP^0\cU^{(jk)}$ is a Riesz basis due to the Fredholm Alternative and Lemma \ref{BiortLemma}.
\end{proof}

\begin{corollary}
The Fr\'echet derivative
\[
d_{V^\dia}\F = \left(d_{V^\dia}\F^{(1)}\,;\, d_{V^\dia}\F^{(2)}\right): \cP^0
\cL^2([0,1];\C^{N\ts N}_\R) \to \cH^{(1)}_\R\oplus \cH^{(2)}_\R,
\]
\[
\cH^{(1)}_\R = \Oplus_{\a=1}^{\a^\dia}\left[\C^{k^\dia_\a\!\ts k^\dia_\a}_\R\!\oplus
\C^{k^\dia_\a\!\ts k^\dia_\a}_\R\!\oplus \C^{(N-k^\dia_\a)\ts
k^\dia_\a}_{\phantom\R}\right]\!,\qquad \cH^{(2)}_\R =
\ell^2_\R\left(\,\N_{n^\dia}\,;\,\C^{N\ts N}_\R\oplus\C^{N\ts N}_\R \right)\!,
\]
is a linear isomorphism (in other words, $d_{V^\dia}\F$ is invertible).
\end{corollary}
\begin{proof}
See Remark \ref{RemarkRiesz} and Proposition \ref{PropRiesz}.
\end{proof}

\subsection{Completion of the proof. Changing of the finite number of first residues.}
\label{SectChangingFirst}

Let $\{(\l_\a^\dia,P_\a^\dag,g_\a^\dag)\}_{\a\ge 1}$ be some data which satisfy conditions
(A)--(C) in Theorem \ref{MainThm}) and $B_\a^\dag=P_\a^\dag (g_\a^\dag)^{-1} P_\a^\dag$.
Recall that $P_{n,j}^\dag\!=\!P_j^0+\ell^2$ and
$(g^\dag_{n,j})^{-1}\!=\!2\pi^2n^2(1+\ell^2_1)$. Similarly to Definition \ref{aceYUSdef}, if
$n$ is sufficiently large, then we may introduce the (unique) factorization
\[
(2\pi^2n^2)^{-1}B_{n,j}^\dag = 
(c_{n,j}^\dag)^2\cdot e_{n,j}^\dag(e_{n,j}^\dag)^*,\quad c_{n,j}^\dag\in\R_+,\
e_{n,j}^\dag\in\C^N,\ \langle e_{n,j}^\dag,e_j^0\rangle =1.
\]
Note that $e_{n,j}^\dag=e_n^0+\ell^2$ and $c_{n,j}^\dag=1+\ell^2$. Define
\[
Y_n^\dag=\left(\begin{array}{ccccccc}c_{n,1}^\dag\cdot e_{n,1}^\dag &;& c_{n,2}^\dag\cdot
e_{n,2}^\dag &; & ... & ; & c_{n,N}^\dag\cdot e_{n,N}^\dag\end{array}\right) \in \C^{N\ts N}.
\]
Since $Y_n^\dag=I_N+\ell^2$, the matrix $Y_n^\dag$ is non-degenerate for all sufficiently
large $n$, and we may introduce its (unique) polar decomposition
\[
Y_n^\dag = U_n^\dag S_n^\dag,\qquad [U_n^\dag]^*=[U_n^\dag]^{-1},\ \ [S_n^\dag]^*=S_n^\dag>0.
\]
Note that $U_n^\dag=I_N+\ell^2$ and $S_n^\dag=I_N+\ell^2$. By our assumptions,
$\sum_{j=1}^NP_{n,j}^\dag=I_N+\ell^2_1$ and $(g_{n,j}^\dag)^{-1}=2\pi^2n^2(1+\ell^2_1)$, so
Lemma \ref{EquivAsymptB} gives
\[
\textstyle U_n^\dag (S_n^\dag)^2 (U_n^\dag)^*= Y_n^\dag(Y_n^\dag)^* =
(2\pi^2n^2)^{-1}\sum_{j=1}^N B_{n,j}^\dag= I_N + \ell^2_1.
\]
Therefore, $S_n^\dag=I_N+\ell^2_1$.

\smallskip

Recall that $U_n^\dia\!=\!U_n(V^\dia)\!=\!I_N$ and $S_n^\dia\!=\!S_n(V^\dia)\!=\!I_N$ for all
$n\ge n^\dia$, so $\F^{(2)}(V^\dia)\!=\!0$. Since the Fr\'echet derivative $d_{V^\dia}\F$ is
invertible, the mapping $\F=(\F^{(1)};\F^{(2)})$ is a local bijection near $V^\dia$.
Therefore, if $\a^\bullet$ is large enough, then there exists some potential $V^\bullet\in
\cB^0_\R(V^\dia,r^\dia)$ such that
\[
\F^{(1)}(V^\bullet)=\F^{(1)}(V^\dia),\quad \F^{(2)}_n(V^\bullet)=\F^{(2)}_n(V^\dia)=0\ \
\mathrm{for\ all}\ n^\dia\le n\le n^\bullet,
\]
\[
\mathrm{and}\qquad \F^{(2)}_n(V^\bullet)=\left(-i\log U^\dag_n\ ;\ 2\pi n \cdot
(S^\dag_n\!-\!I_N)\right)\quad \mathrm{for\ all}\ n\ge n^\bullet,
\]
where $\a^\bullet-\a^\dia = N(n^\bullet-n^\dia)$ (i.e., $\a^\bullet+1$ corresponds to the
double-index $(n^\bullet,1)$). Since the original mapping $\wt\F$ can be reconstructed from
$\F$, one has
\[
\wt{A}_\a(V^\bullet)=\wt{A}_\a(V^\dia)=0\quad \mathrm{for\ all}\ \a\le\a^\bullet,
\]
\[
\wt{A}_{n,j}(V^\bullet)=0\quad \mathrm{and}\quad \wt{B}_{n,j}(V^\bullet)=B_{n,j}^\dag \qquad
\mathrm{for\ all}\ n\ge n^\bullet.
\]
Due to Lemma \ref{ABselfadjoint}, it gives
\[
\s(V^\bullet)=\{\l_\a^\dia\}_{\a\ge 1}\qquad \mathrm{and}\qquad
B_{n,j}(V^\bullet)=B_{n,j}^\dag \ \ \mathrm{for\ all}\ n\ge n^\bullet.
\]

\smallskip

At last, we need to change the {\it finite} number of first residues
$(B_\a(V^\bullet))_{\a=1}^{\a^\bullet}$ to $(B_\a^\dag)_{\a=1}^{\a^\bullet}$. Recall that the
isospectral transforms constructed in \cite{CK} allow to modify each particular residue $B_\a$
in an almost arbitrary way. The only one restriction (concerning the change of projector
$P_\a$ to $\wt{P}_\a$) is
\[
\cF_\a\cap\Ran \wt{P}_\a = \{0\},
\]
where $\cF_\a$, $\dim\cF_\a=N\!-\!k_\a$ is some "forbidden" subspace that is uniquely
determined by the spectrum and all other subspaces $(\cE_\b)_{\b\ne\a}$. It's not hard to
conclude (see Proposition \ref{FaAnd(C)}) that this restriction is {\it equivalent} to the
following:
\begin{quotation}
One can modify $B_\a$ in an arbitrary way such that (C) holds true.
\end{quotation}
In general situation one can change all $B_\a(V^\bullet)$ to $B_\a^\dag$ by $\a^\bullet$
steps. Nevertheless, it may happen that at some intermediate step the desired residue
$B_\a^\dag$ violates (C). In order to overcome this difficulty note that one can always change
$B_\a$ to some $\wt{B}_\a^\dag$ which is arbitrary close to $B_\a^\dag$ in the natural
topology. Then, in any case, after $\a^\bullet$ steps one can obtain some potential
$\wt{V}^\bullet$ such that $B_\a(\wt{V}^\bullet)=\wt{B}_\a^\dag$  for all $\a=1,..,\a^\bullet$
(and, of course, $B_\a(\wt{V}^\bullet)=B_\a(V^\bullet)=B_\a^\dag$ for all $\a>\a^\bullet$). By
Corollary \ref{RemOpenSet}, the set of all admitted by (C) sequences
$(B_\a)_{\a=1}^{\a^\bullet}$ is open in the natural topology. Therefore, if
$(\wt{B}_\a^\dag)_{\a=1}^{\a^\bullet}$ and $(B_\a^\dag)_{\a=1}^{\a^\bullet}$ are close enough,
then all changes $\wt{B}_\a^\dag\mapsto B_\a^\dag$ are permitted. So, after another at most
$\a^\bullet$ steps one obtains the potential $V$ such that $B_\a(V)=B_\a^\dag$ for all
$\a=1,..,\a^\bullet$ (and still $B_\a(V)=B_\a^\dag$ for all $\a>\a^\bullet$). The proof is
finished. \qed

\smallskip

\renewcommand{\thesection}{A}
\section{Appendix. Property (C)}
\setcounter{equation}{0}
\renewcommand{\theequation}{A.\arabic{equation}}
\label{AppendixA}

Let $\l_\a>0$ for all $\a\ge 1$. Note that (C) doesn't depend on shifts of the spectrum, so we
do not lose the generality.  We begin with the following simple

\begin{remark}
If an entire function $\x$ is bounded on the real positive half-line, then the condition
$\x(\l)=O(e^{|\Im\sqrt\l|})$ is equivalent to say that $\x(z^2)$ is an entire function of
exponential type no greater than $1$ (see \cite{Ko}, p.28).
\end{remark}

Recall that the Paley-Wiener space $PW_{[-1,1]}$ consists of all entire functions $f(z)$ of
exponential type no greater than $1$ such that $f\in \cL^2(\R)$. The Paley-Wiener theorem (see
\cite{Ko} p.30) claims
\begin{equation}
\label{PW} f\in PW_{[-1,1]}\qquad\mathrm{iff}\qquad
f(z)=\frac{1}{2\pi}\int_{-1}^1\f(t)e^{-izt}dt,\quad \mathrm{where}\ \ \f\in \cL^2(-1,1).
\end{equation}

\begin{proof}[\bf Proof of Proposition \ref{Creform}] If $\f\in \cL^2([-1,1]\,;\C^N)$ is some vector-valued function
such that
\begin{equation}
\label{xxOrt} \int_{-1}^1\phi(t)dt=0\qquad \mathrm{and}\qquad h_\a^*\int_{-1}^1\phi(t)e^{\pm
i\sqrt{\l_\a}t} dt =0\ \ \mathrm{for\ all}\ \a\ge 1,
\end{equation}
then
\[
\frac{1}{2\pi}\int_{-1}^1\f(t)e^{-izt}dt=zf(z)\quad \mathrm{and}\quad P_\a
f(\pm\sqrt{\l_\a})=0\ \ \mathrm{for\ all}\ \a\ge 1,
\]
where $zf(z)\in PW_{[-1,1]}$. Denote $\xi(z^2)=\frac{1}{2}[f(z)\!+\!f(-z)]$ or
$\xi(z^2)=\frac{1}{2z}[f(z)\!-\!f(-z)]$. Then, $P_\a\x(\l_\a)\!=\!0$, $\a\ge 1$,
$\x(\l)\!=\!O(e^{|\Im\sqrt\l|})$ and $\x\!\in\!\cL^2(\R_+)$. This contradicts to (C).

Conversely, let $\x(\l)\!=\!O(e^{|\Im\sqrt\l|})$ and $\x\in \cL^2(\R_+)$. Then
$f(z)\!=\!z\x(z^2)\in PW_{[-1,1]}$, so it admits representation (\ref{PW}) with some $\f\in
\cL^2(-1,1)$. It's easy to check that $P_\a\x(\l_\a)=0$ and $f(0)=0$ imply (\ref{xxOrt}).
Hence, $\f\equiv 0$.
\end{proof}

We have the immediate

\begin{corollary}
\label{RemOpenSet} If one fixes the spectrum $\{\l_\a\}_{\a\ge 1}$ and all projectors
$P_\a$,~$\a\ge\a^\bullet\!+\!1$, for some $\a^\bullet\ge 0$, then the set of all finite
sequences $(P_\a)_{\a=1}^{\a^\bullet}$ satisfying the condition (C) is open in the natural
topology.
\end{corollary}

Introduce the function
\begin{equation}
\label{xibDef} \xi_\b(\l)\equiv\frac{\c(0,\l,V)P_\b^\sharp}{\l\!-\!\l_\b},
\end{equation}
where $P_\b^\sharp:\C^N\to\cE_\b^\sharp$ is the orthogonal projector onto the subspace
$\cE_\b^\sharp=\Ker\c(0,\l_\b,V)$.

\begin{proposition}
\label{xibProp} Let $\b\ge 1$ and $V\!=\!V^*\in \cL^2([0,1];\C^{N\ts N}_\R)$. Then,

\smallskip

\noindent (i) $\x_\b:\C\to\C^{N\ts N}$ is an entire matrix-valued function,
$\x_\b(\l)=O(e^{|\Im\sqrt\l|})$ as $|\l|\to\iy$, $\x_\b\in\cL^2(\R_+)$ and
$P_\a\x_\b(\l_\a)=0$ for all $\a\ne\b$.

\smallskip

\noindent (ii) If $\xi:\C\to \C^N$ is an entire vector-valued function such that
$\xi(\lambda)\!=\!O(e^{|\Im\sqrt\lambda|})$ as $|\lambda|\to\infty$, $\xi\in \cL^2(\R_+)$ and
$P_\a\xi(\lambda_\a)\!=\!0$ for all $\a\ne\b$, then $\xi=\xi_\b h$ for some $h\in\C^N$.
\end{proposition}

\begin{proof}
(i) The function $\x_\b$ is entire due to $\c(0,\l_\b,V)P_\b^\sharp=0$. Furthermore,
\[
\x_\b(\l)=O(|\l|^{-\frac{3}{2}}e^{|\Im\sqrt\l|})\ \ \mathrm{as}\
|\l|\to\infty\qquad\mathrm{and}\qquad P_\a \x_\b(\l_\a) = 0\ \ \mathrm{for\ all}\ \a\ne\b,
\]
since $P_\a\c(0,\l_\a)=P_\a[\vp(1,\l_\a)]^* = [\vp(1,\l_\a)P_\a]^*=0$.

\smallskip

\noindent (ii) Lemma 2.2 \cite{CK} claims
\[
[\c(0,\l,V)]^{-1}= [\vp^*(1,\ol{\l},V)]^{-1}=
(Z_\a^{-1}+O(\l\!-\!\l_\a))((\l\!-\!\l_\a)^{-1}P_\a + P_\a^\perp)\ \ \mathrm{as}\ \ \l\to\l_\a
\]
for some $Z_\a$, $\a\ne\b$, such that $\det Z_\a\ne 0$ and
\[
[\c(0,\l,V)]^{-1}= [\vp(1,\l,V^\sharp)]^{-1}= ((\l\!-\!\l_\b)^{-1}P_\b^\sharp +
(P_\b^\sharp)^\perp)(Z_\b^{-1}+O(\l\!-\!\l_\b))\ \ \mathrm{as}\ \ \l\to\l_\b
\]
for some $Z_\b$, $\det Z_\b\ne 0$. Due to $P_\a\x(\l_\a)=0$, $\a\ne\b$, the (vector-valued)
function
\[
\o(\l)=[\c(0,\l,V)]^{-1}\x(\l)
\]
is analytic except $\l_\b$ and $\o(\l)=(\l\!-\!\l_\b)^{-1}P_\b^\sharp h + O(1)$ as
$\l\to\l_\b$ for some $h\in\C^N$. Since $\o(\l)=O(|\l|^{1/2})$ as
$|\l|=\pi^2(n+\frac{1}{2})^2\to\iy$, the Liouville theorem gives
\[
\xi(\l) \equiv \c(0,\l,V)\left[ (\l\!-\!\l_\b)^{-1}P_\b^\sharp h + \o_0 \right] \equiv
\x_\b(\l) h + \c(0,\l,V)\o_0\quad \mathrm{for\ some}\ \o_0\in\C^N.
\]
Finally, $\x\in L^2(\R_+)$ implies $\o_0=0$.
\end{proof}

Recall the construction of the "forbidden" subspaces $\cF_\a\ss\C^N$, $\a\ge 1$, given in
\cite{CK}. Let $V\!=\!V^*\in \cL^2([0,1];\C^{N\ts N})$. For each $\a\!\ge\!1$ denote
\[
\cF_\a=[S_\a(\cE_\a)]^\perp,\quad \mathrm{where}\ \
S_\a=S_\a(V)=\int_0^1[\vp^*\vp](t,\l_\a,V)dt=S_\a^*>0
\]
and $\cE_\a=\Ran P_\a$. Note that $\dim\cF_\a=N\!-\!\dim\cE_\a=N\!-\!k_\a$. The main result of
\cite{CK} is that one can modify each particular projector $P_\a$ (keeping the spectrum and
all other projectors fixed) in an arbitrary way such that $\cF_\a\cap \Ran P_\a =\{0\}$. It's
quite natural that this restriction is equivalent to property (C) as shows

\begin{proposition}[{\bf Connection between subspaces $\cF_\a$ and property (C)}]
\label{FaAnd(C)} $\phantom{\quad}$\linebreak Let $\b\ge 1$ and
$(\l_\a\,;P_\a)_{\a=1}^{+\infty}= (\l_\a(V)\,;P_\a(V))_{\a=1}^{+\infty}$ for some
$V\!=\!V^*\in \cL^2([0,1];\C^{N\ts N}_\R)$. Then, the collection
$(\lambda_\a\,;\wt{P}_\a)_{\a=1}^{+\infty}$, where $\wt{P}_\a=P_\a$ for all $\a\ne \b$,
satisfies (C) iff
\begin{equation}
\label{ForbiddenCond} \cF_\b\cap\Ran \wt{P}_\b =\{0\},\quad\mathit{where}\quad
\cF_\b=[S_\b(\cE_\b)]^\perp.
\end{equation}
Moreover, $\cF_\b=[\Ran\x_\b(\l_\b)]^\perp$, where $\x_\b$ is given by (\ref{xibDef}).
\end{proposition}

\begin{proof}
It follows from Proposition \ref{xibProp} (ii) that (C) holds true for the new collection
$(\lambda_\a\,;\wt{P}_\a)_{\a=1}^{+\infty}$ if and only if $\wt P_\b \x_\b(\l_\b)h\ne 0$ for
all $h\in\cE_\b^\sharp$, $h\ne 0$. In other words, (C) is equivalent to
\begin{equation}
\label{xxx} \Ran\x_\b(\l_\b)\cap \Ker\wt{P}_\b=\{0\}.
\end{equation}
One has (see Lemmas 2.4 and 2.1 \cite{CK} for details)
\[
\x_\b(\l_\b)=\dot\c(0,\l_\b)P_\b^\sharp = \dot\vp^*(1,\l_\b)P_\b^\sharp =
-\dot\vp^*(1,\l_\b)\vp'(1,\l_\b)\c'(0,\l_\b)P_\b^\sharp.
\]
Moreover, $\Ran \c'(0,\l_\b)P_\b^\sharp= \cE_\b$ and
\[
\Ran \x_\b(\l_\b) = \Ran [\dot\vp^*(1,\l_\b)\vp'(1,\l_\b)P_\b] = \Ran S_\b P_\b =
S_\b(\cE_\b).
\]
Since $\dim \Ker \wt{P}_\b= N-k_\b = N-\dim S_\b(\cE_\b)$, (\ref{xxx}) is equivalent to
(\ref{ForbiddenCond}).
\end{proof}

We finish our discussion by the consideration of the special case when only finite number of
$P_\a$ differ from the standard unperturbed coordinate projectors.

\begin{quotation}
\noindent { Let $A=\{\a_1,\a_2,..,\a_m\}$ be some finite set of exceptional indices. Assume
that $P_\a=P_\a^0$ coincides with some {\it coordinate} projector $P_\a^0$ for all $\a\notin
A$ (we admit multiple eigenvalues). Introduce the sets
\[
A_j^0 = \{\a\notin A: P_\a e_j^0\ne 0\}
\]
(possible multiple eigenvalues belong to several $A_j^0$). Assume that there exists $C>0$ such
that the set $\{\l_\a,\a\in A_j^0\}\cap \left(-\infty, \pi^2n^2+C\right]$ consists of exactly
$n\!-\!m$ points for all $j=1,2,..,N$, if $n$ is large enough. Let
\[
k_{\a_1}+k_{\a_2}+..+k_{\a_m}=Nm
\]
}
\end{quotation}
We give the simple description of all {\it finite} sequences $(P_{\a_s})_{s=1}^m$, $\rank
P_\a=k_\a$, such that the whole collection $\{(\lambda_\a\,;P_\a)\}_{\a=1}^{+\infty}$
satisfies (C):

\begin{proposition}
\label{FiniteProj} Let $(\lambda_\a\,;P_\a)_{\a=1}^{+\infty}$ be as described above. Then (C)
holds true iff
\[
\cT=\left(\begin{array}{cccc} T_0 & T_1 & ... & T_{m-1} \cr T_1 & T_2 & ... & T_m \cr ... &
... & ... & ... \cr T_{m-1} & T_m & ... & T_{2m-2}
\end{array}\right)=\cT^*>0,
\]
where
\[
T_k=\sum_{\a\in A} \l_\a^k F(\l_\a)P_\a F(\l_\a)=T_k^*,\qquad k=0,1,..,2m\!-\!2,
\]
\[
F(\l)\equiv\diag\{f_1(\l),f_2(\l),..,f_N(\l)\}\quad\mathit{and}\quad f_j(\l)\equiv
\prod\nolimits_{\a\in A_j^0} \lt(1-\frac{\l}{\l_\a}\rt).
\]
\end{proposition}

\begin{remark}
Since $\cT\ge 0$ in any case, the condition $\cT>0$ is equivalent to $\det\cT\ne 0$.
\end{remark}

\begin{proof} Indeed, let $\x(\l)=\left(\x_1(\l),\x_2(\l),..,\x_N(\l)\right)^\top$ be such
that $\x(\l)=O(e^{|\Im\sqrt\l|})$, $\x\in \cL^2(\R_+)$ and $P_\a\x(\l_\a)=0$ for all $\a\ge
1$. In particular, $P_j^0\x(\l_\a)=0$ for all $\a\in A_j^0$. In order words, $z\x_j(z^2)\in
PW_{[-1,1]}$ and $\x_j(\l_\a)=0$ for all $\a\in A_j^0$. Therefore,
\[
\x_j(\l)\equiv Q_j(\l)f_j(\l),\qquad \deg Q_j\le m\!-\!1,
\]
for some polynomials $Q_j$. Let
\[
Q(\l)=(Q_1(\l),Q_2(\l),..,Q_N(\l))^\top=\sum_{p=0}^{m-1}\l^py_p,\ \ y_p\in \C^N \ \
\mathrm{and}\ \ y=\left(\vphantom{\big|}y_p\right)_{p=0}^{m-1}\in \C^{Nm}.
\]
Then,
\[
y^*\cT y = \sum_{p,q=0}^{m-1}y_p^*T_{p+q}y_q = \sum_{p,q=0}^{m-1}y_p^* \lt[\sum_{\a\in A}
\l_\a^{p+q} F(\l_\a) P_\a F(\l_\a)\rt] y_q
\]
\[
=\sum_{a\in A}\, [Q(\l_\a)]^* F(\l_\a) P_\a F(\l_\a) Q(\l_\a) = \sum_{\a\in A}\, [\x(\l_\a)]^*
P_\a \x(\l_\a).\vphantom{\sum^|}
\]
Hence, the $Nm\!\ts\!Nm$ matrix $\cT$ is degenerate iff there exists $\x$ such that $P_\a
\x(\l_\a)=0$ for all $\a\in A$ (recall that $P_\a^0\x(\l_\a)=0$ holds true for all $\a\notin
A$ by the construction).
\end{proof}

\renewcommand{\thesection}{B}
\section{Appendix. Three classical choices of additional spectral data in the scalar case.}
\setcounter{equation}{0}
\renewcommand{\theequation}{B.\arabic{equation}}
\label{AppendixB}

In the scalar case, it is well known that the Dirichlet spectrum
$\s(q)=\{\l_n(q)\}_{n=1}^{+\iy}$ determines only "one half" of the potential $q$. Thus, in
other to determine $q$ uniquely, one needs either to assume that some partial information
about $q$ is known or to consider some additional spectral data besides $\s(q)$. Note that
there are two classical assumptions about the potential that make the knowledge of the
spectrum sufficient: symmetry $q(x)\equiv q(1\!-\!x)$ (see, e.g., \cite{PT}) or the knowledge
of $q(x)$ as $x\in [0,\frac{1}{2}]$ (the~Hochstadt-Lieberman theorem \cite{HL}, see also
\cite{GS}, \cite{Ho}, \cite{MP}). Also, there are several classical choices of additional
spectral data:
\begin{enumerate}
\item The second spectrum. This setup goes back to the original paper of Borg \cite{Bo}.
The most natural choice is the spectrum $\{\m_n(q)\}_{n=1}$ of the mixed problem
\[
-y''+qy=\l y,\qquad y(0)=y'(1)=0.
\]
Note that $\{\m_n(q)\}_{n=1}^{+\iy}\cup\{\l_n(q)\}_{n=1}^{+\iy}$ is the Dirichlet spectrum of
the symmetric potential $q(2-x)\equiv q(x)$, $x\in [0,1]$, defined on the doubled interval
$[0,2]$.
\item The normalizing constants (firstly appeared in Marchenko's paper \cite{Ma1})
\[
[\a_n(q)]^{-1}=\lt[\int_0^1\vp^2(x,\l_n)dx\rt]^{-1}=[\dot\vp\vp']^{-1}(1,\l_n)=
-\frac{\c'(0,\l_n)}{\dot\c(0,\l_n)} = -\res_{\l=\l_n} m(\l).
\]
\item The norming constants introduced by Trubowitz and co-authors (see \cite{PT})
\[
\n_n(q)=\log[(-1)^n\vp'(1,\l_n)]=\log\lt[(-1)^{n-1}\frac{\vp(\cdot,\l_n)}{\c(\cdot,\l_n)}\rt].
\]
\end{enumerate}
It is quite well known in the folklore that the characterization problems in the setups
(1)-(3) are equivalent. Unfortunately, we do not know the good reference for this fact. So,
the main purpose of this Appendix is to give the short proof of these equivalences (note that
our arguments are quite similar to \cite{Lev}). For the simplicity, we assume that $q\in
L^2(0,1)$, $\int_0^1 q(x)dx=0$, i.e., $\{\l_n(q)-\pi^2n^2\}_{n=1}^{+\iy}\in\ell^2$ (the
similar arguments work well for other classes of potentials and corresponding classes of
spectral data).

Note that
\begin{equation}
\label{xML} \m_1<\l_1<\m_2<\l_2<\m_2<...\quad \mathrm{and}\quad
\m_n=\pi^2(n-{\textstyle\frac{1}{2}})^2+O(1)
\ \ \mathrm{as}\ \ n\to\iy.
\end{equation}
Also, the Hadamard factorization implies
\begin{equation}
\label{xVpDef} f(\l)=\vp(1,\l)=
\prod_{m=1}^{+\iy}\frac{\l_m\!-\!\l}{\pi^2m^2}\qquad \mathrm{and}\qquad g(\l)=\vp'(1,\l)=
\prod_{m=1}^{+\iy}\frac{\m_m\!-\!\l}{\pi^2(m\!+\!\frac{1}{2})^2}.
\end{equation}
Recall that we write $a_n=b_n+\ell^2_k$ iff $\{n^k|a_n\!-\!b_n|\}_{n=1}^{+\iy}\in\ell^2$.

\begin{proposition} Let $\l_n=\pi^2n^2+\ell^2$, (\ref{xML}) hold and $f(\l)$, $g(\l)$ be
given by (\ref{xVpDef}). Then, the following conditions are equivalent:

(1) The asymptotics $\m_n=\pi^2(n\!-\!\frac{1}{2})^2+\ell^2\vphantom{\Big|}$ hold true.

(2) The asymptotics $\a_n=g(\l_n)\dot{f}(\l_n)=(2\pi^2n^2)^{-1}(1+\ell^2_1)\vphantom{\Big|}$
hold true.

(3) The asymptotics $\n_n=\log[(-1)^n g(\l_n)]=\ell^2_1\vphantom{\Big|}$ hold true.
\end{proposition}

\begin{proof}
We start with the equivalence $\mathrm{(2)}\Leftrightarrow\mathrm{(3)}$. Denote
$\wt{\l}_n=\pi^{-2}\l_n-n^2=O(1)$ as $n\to\iy$. Then
\[
\dot{f}(\l_n) = -\frac{1}{\pi^2n^2}\prod_{m\ne n}\frac{\l_m\!-\!\l_n}{\pi^2m^2}=
\frac{(-1)^{n}}{2\pi^2n^2}\prod_{m\ne
n}\frac{\l_m\!-\!\l_n}{\pi^2(m^2\!-\!n^2)}=\frac{(-1)^{n}}{2\pi^2n^2}\prod_{m\ne
n}\lt[1+\frac{\wt{\l}_m\!-\!\wt{\l}_n}{m^2\!-\!n^2}\rt].
\]
Note that
\[
\begin{array}{l}
\displaystyle\log\prod_{m\ne n}\lt[1+\frac{\wt{\l}_m\!-\!\wt{\l}_n}{m^2\!-\!n^2}\rt]=
\sum_{m\ne n}
\lt[\frac{\wt{\l}_m\!-\!\wt{\l}_n}{m^2\!-\!n^2}+O\lt(\frac{1}{(m^2\!-\!n^2)^2}\rt)\rt] \cr
\displaystyle \hphantom{\log\prod_{m\ne
n}\lt[1+\frac{\wt{\l}_m\!-\!\wt{\l}_n}{m^2\!-\!n^2}\rt]}= \sum_{m\ne
n}\frac{\wt{\l}_m\!-\!\wt{\l}_n}{m^2\!-\!n^2}+O\lt(\frac{1}{n^2}\rt)= \sum_{m\ne
n}\frac{\wt{\l}_m}{m^2\!-\!n^2}+O\lt(\frac{1}{n^2}\rt).
\end{array}
\]
Then, it immediately follows from $(\wt{\l}_n)_{n=1}^{+\iy}\in\ell^2$ and simple properties of
the discrete Hilbert transform (see Lemma \ref{DiscrHilbert} (ii) below) that
$\dot{f}(\l_n)=(-1)^n(2\pi^2n^2)^{-1}(1+\ell^2_1)$. Thus,
$\mathrm{(2)}\Leftrightarrow\mathrm{(3)}$. The proof of the equivalence
$\mathrm{(1)}\Leftrightarrow\mathrm{(3)}$ is similar. Indeed,
\[
g(\l_n)=\prod_{m=1}^{+\iy}\frac{\m_m\!-\!\l_n}{\pi^2(m\!-\!\frac{1}{2})^2}=
(-1)^n\!\prod_{m=1}^{+\iy}\frac{\m_m\!-\!\l_n}{\pi^2\left((m\!-\!\frac{1}{2})^2\!-\!n^2\right)}=
(-1)^n\!\prod_{m=1}^{+\iy}\lt[1+\frac{\wt{\m}_m\!-\!\wt{\l}_n}{(m\!-\!\frac{1}{2})^2\!-\!n^2}\rt],
\]
where $\wt{\m}_m=\pi^{-2}\m_m-(m+\frac{1}{2})^2=O(1)$ as $m\to\iy$. As above,
\[
\log[(-1)^ng(\l_n)]=\sum_{m=1}^{+\iy}\frac{\wt{\m}_m\!-\!\wt{\l}_n}{(m\!-\!\frac{1}{2})^2\!-\!n^2}
+ O\lt(\frac{1}{n^2}\rt) = \sum_{m=1}^{+\iy}\frac{\wt{\m}_m}{(m\!-\!\frac{1}{2})^2\!-\!n^2} +
O\lt(\frac{1}{n^2}\rt)
\]
and the equivalence $\mathrm{(1)}\Leftrightarrow\mathrm{(3)}$ follows by Lemma
\ref{DiscrHilbert} (i).
\end{proof}

\begin{lemma} \label{DiscrHilbert} (i) The linear operator $(a_m)_{m=1}^{+\iy}\mapsto
(b_n)_{n=1}^{+\iy}$, where
\[
b_n=\frac{1}{2\pi n}\sum_{m=1}^{+\iy}\frac{a_m}{n^2\!-\!(m\!-\!\frac{1}{2})^2}=
\frac{1}{\pi}\sum_{m=1}^{+\iy}
\lt[\frac{a_m}{n\!-\!m\!+\!\frac{1}{2}}+\frac{a_m}{n\!-\!(1\!-\!m)\!+\!\frac{1}{2}}\rt],
\]
is an isometry in $\ell^2$.

\noindent (ii) The linear operator $(a_m)_{m=1}^{+\iy}\mapsto (b_n)_{n=1}^{+\iy}$, where
\[
b_n=\frac{1}{2n}\sum_{m=1}^{+\iy}\frac{a_m}{n^2\!-\!m^2}=
\sum_{m=1}^{+\iy}\lt[\frac{a_m}{n\!-\!m}+\frac{a_m}{n\!-\!(-m)}\rt],
\]
is bounded in $\ell^2$.
\end{lemma}

\begin{proof}
Both results easily follows by the Fourier transform and the identities (in $L^2({\mathbb
T})$)
\[
\sum_{k=-\iy}^{+\iy}\frac{\z^k}{k+\frac{1}{2}}=\frac{\pi i}{\sqrt\z}=\pi i
e^{-\frac{i\phi}{2}}\qquad\mathrm{and}\qquad \sum_{k\ne 0}\frac{\z^k}{k}=-i(\phi-\pi),
\]
where $\z=e^{i\phi}\ne 1$, $\phi\in (0,2\pi)$.\end{proof}

\begin{remark}
The similar technique can be applied for other inverse problems in order to derive the
characterization of some additional spectral parameters (e.g., similar to $\a_n(q)$) from the
characterization of other parameters (e.g., similar to $\n_n(q)$). In general, these
characterizations may differ from each other substantially, see \cite{CK1}.
\end{remark}


\begin{thebibliography}{CHGL00}

\bibitem [AM63] {AM}  Agranovich, Z. S.; Marchenko, V. A. The inverse problem of scattering theory.
Translated from the Russian by B. D. Seckler. Gordon and Breach Science Publishers, New
York-London 1963 xiii+291 pp.

\bibitem [Bo46] {Bo}  Borg, G. Eine Umkehrung der Sturm-Liouvilleschen Eigenwertaufgabe. Bestimmung der Differentialgleichung
durch die Eigenwerte. (German) Acta Math. 78, (1946). 1--96.


\bibitem [CD76] {CD1} Calogero, F.; Degasperis, A. Nonlinear evolution equations solvable by the inverse spectral transform.
I. Nuovo Cimento B (11) 32 (1976), no. 2, 201--242.

\bibitem [CD77] {CD2} Calogero, F.; Degasperis, A. Nonlinear evolution equations solvable by the inverse spectral transform.
II. Nuovo Cimento B (11) 39 (1977), no. 1, 1--54.

\bibitem [Car02] {Ca} Carlson, R. An inverse problem for the matrix Schrodinger equation.
J. Math. Anal. Appl. 267 (2002), no. 2, 564--575.

\bibitem [CHGL00] {CHGL} Clark, S.; Gesztesy, F.; Holden, H.; Levitan, B. M.
Borg-type theorems for matrix-valued Schrodinger operators. J. Differential Equations 167
(2000), no. 1, 181--210.

\bibitem [CKK04] {CKK} Chelkak, D.; Kargaev, P.; Korotyaev, E. Inverse problem for harmonic oscillator perturbed by potential,
characterization. Comm. Math. Phys. 249 (2004), no. 1, 133--196.

\bibitem [CK06a] {CKper} Chelkak, D.; Korotyaev, E.
Spectral estimates for Schr\"odinger operators with periodic matrix potentials on the real
line. Int. Math. Res. Not. 2006, Art. ID 60314, 41 pp.

\bibitem [CK06b] {CK} Chelkak, D.; Korotyaev, E. Parametrization of the isospectral set
for the vector-valued Sturm-Liouville problem. J. Funct. Anal. 241 (2006), no. 1, 359--373.

\bibitem [CK07] {CK1} Chelkak, D.; Korotyaev, E. The inverse problem for perturbed harmonic oscillator on the half-line
with a Dirichlet boundary condition. Ann. Henri Poincare 8 (2007), no. 6, 1115--1150.

\bibitem [ChSh97] {ChSh} Chern, Hua-Huai; Shen, Chao-Liang
On the $n$-dimensional Ambarzumyan's theorem. Inverse Problems 13 (1997), no. 1, 15--18.

\bibitem [DT84] {DT} Dahlberg, B. E. J.; Trubowitz, E. The inverse Sturm-Liouville problem. III.
Comm. Pure Appl. Math. 37 (1984), no. 2, 255--267.

\bibitem [Di99] {Di} Dineen, S. Complex analysis on infinite-dimensional spaces.
Springer Monographs in Mathematics. Springer-Verlag London, Ltd., London, 1999. xvi+543 pp.

\bibitem [FM76] {FM} Flaschka, H.; McLaughlin, D. W. Canonically conjugate variables for the Korteweg-de Vries equation
and the Toda lattice with periodic boundary conditions. Progr. Theoret. Phys. 55 (1976), no. 2, 438--456.

\bibitem [FY01] {FY} Freiling, G.; Yurko, V. Inverse Sturm-Liouville problems and their applications.
Nova Science Publishers, Inc., Huntington, NY, 2001. x+356 pp.

\bibitem [GL51] {GL} Gel'fand, I. M.; Levitan, B. M. On the determination of a differential equation
from its spectral function. (Russian) Izvestiya Akad. Nauk SSSR. Ser. Mat. 15, (1951).
309--360. English Translation: Gel'fand, I. M.; Levitan, B. M. On the determination of a
differential equation from its spectral function. Amer. Math. Soc. Transl. (2) 1 (1955),
253--304.

\bibitem [Ges07] {Ges} Gesztesy, F.
Gesztesy, F. Inverse spectral theory as influenced by Barry Simon. Spectral theory and
mathematical physics: a Festschrift in honor of Barry Simon's 60th birthday, 741--820, Proc.
Sympos. Pure Math., 76, Part 2, Amer. Math. Soc., Providence, RI, 2007.

\bibitem [GS00] {GS} Gesztesy, F.; Simon, B. Inverse spectral analysis with partial
information on the potential. II. The case of discrete spectrum. Trans. Amer. Math. Soc. 352
(2000), no. 6, 2765--2787.

\bibitem [GR88] {GR} Guillot J.G., Ralston J.V.: Inverse spectral theory for a
singular Sturm-Liouville operator on $[0,1]$, J. Diff. Eq., 76 (1988), 353--373.

\bibitem [HL78] {HL} Hochstadt, H.; Lieberman, B.
An inverse Sturm-Liouville problem with mixed given data. SIAM J. Appl. Math. 34 (1978), no.
4, 676--680.

\bibitem [Ho05] {Ho} Horvath, M. Inverse spectral problems and closed exponential systems. Ann. of Math. (2)
162 (2005), no. 2, 885--918.

\bibitem [IMT84] {IMT} Isaacson, E.; McKean, H.; Trubowitz, E.:
The inverse Sturm-Liouville problem. II. Comm. Pure Appl. Math. 37(1984), no. 1, 1--11.

\bibitem [IT83] {IT} Isaacson, E.; Trubowitz, E.
The inverse Sturm-Liouville problem. I. Comm. Pure Appl. Math. 36 (1983), no. 6,
   767--783.

\bibitem [JL98a] {JL1} Jodeit, M. Jr.; Levitan, B. M. Isospectral vector-valued Sturm-Liouville problems.
Lett. Math. Phys. 43 (1998), no. 2, 117--122.

\bibitem [JL98b] {JL2} Jodeit, Max, Jr.; Levitan, B. M. A characterization of some even vector-valued
Sturm-Liouville problems. Mat. Fiz. Anal. Geom. 5 (1998), no. 3-4, 166--181.

\bibitem [Ko88] {Ko} Koosis, P. The logarithmic integral. I. Cambridge Studies in Advanced Mathematics, 12.
Cambridge University Press, Cambridge, 1988. xvi+606 pp.

\bibitem [Kr51] {Kr1}
Kre\u \i n, M. G. Solution of the inverse Sturm-Liouville problem. (Russian) Doklady Akad.
Nauk SSSR (N.S.) 76, (1951). 21--24.
\bibitem [Kr53] {Kr1a}
Kre\u\i n, M. G. On the transfer function of a one-dimensional boundary problem of the second
order. (Russian) Doklady Akad. Nauk SSSR (N.S.) 88, (1953). 405--408.
\bibitem [Kr54] {Kr2}
Kre\u\i n, M. G. On a method of effective solution of an inverse boundary problem. (Russian)
Doklady Akad. Nauk SSSR (N.S.) 94, (1954). 987--990.

\bibitem [Le40] {Le} Levinson, N. Gap and Density Theorems. American Mathematical Society Colloquium
Publications, v. 26. American Mathematical Society, New York, 1940.

\bibitem [Le49] {Levinson} Levinson, N. The inverse Sturm-Liouville problem. Mat. Tidsskr. B. 1949, (1949). 25--30.

\bibitem [Lev64] {Lev} Levitan, B. M. Determination of a Sturm-Liouville differential equation in terms of two
spectra. (Russian) Izv. Akad. Nauk SSSR Ser. Mat. 28 (1964) 63--78.

\bibitem [Lev87] {LeBook}
Levitan, B. M. Inverse Sturm-Liouville problems. Translated from the Russian by O. Efimov.
VSP, Zeist, 1987. x+240 pp.

\bibitem [LG64] {LG}  Levitan, B. M.; Gasymov, M. G. Determination of a differential equation
by two spectra. (Russian) Uspehi Mat. Nauk 19 1964 no. 2 (116), 3--63.

\bibitem [MP05] {MP} Makarov, N.; Poltoratski, A. Meromorphic inner functions,
Toeplitz kernels and the uncertainty principle. Perspectives in analysis, 185--252, Math.
Phys. Stud., 27, Springer, Berlin, 2005.

\bibitem [Mal05] {Mal} Malamud, M. M. Uniqueness of the matrix Sturm-Liouville equation given
a part of the monodromy matrix, and Borg type results. Sturm-Liouville theory, 237--270, Birkhauser, Basel, 2005.

\bibitem [Mar50] {Ma1}  Mar\v cenko, V. A. Concerning the theory of a differential operator of the second order.
(Russian) Doklady Akad. Nauk SSSR. (N.S.) 72, (1950). 457--460.

\bibitem [Mar86] {MaBook} Marchenko, V. A. Sturm-Liouville operators and applications. Translated from the Russian by A. Iacob.
Operator Theory: Advances and Applications, 22. Birkhauser Verlag, Basel, 1986. xii+367 pp.

\bibitem [MO75] {MO} Mar\v cenko, V. A.; Ostrovski\u\i, I. V. A characterization of the spectrum of the Hill operator.
(Russian) Mat. Sb. (N.S.) 97(139) (1975), no. 4(8), 540--606, 633--634.

\bibitem [MT81] {MT} McKean, H. P.; Trubowitz, E. The spectral class of the quantum-mechanical harmonic oscillator.
Comm. Math. Phys. 82 (1981/82), no. 4, 471--495.

\bibitem [Ol85] {O}  Olmedilla, E. Inverse scattering transform for general matrix Schr\"odinger operators
and the related symplectic structure. Inverse Problems 1 (1985), no. 3, 219--236.


\bibitem [PT87] {PT} P\"oschel, J.; Trubowitz, E. Inverse spectral theory. Pure and Applied Mathematics, 130.
Academic Press, Inc., Boston, MA, 1987. x+192 pp.

\bibitem [SP04] {SP} Samsonov, B. F.; Pecheritsin, A. A. Chains of Darboux transformations for the matrix
Schr\"odinger equation. J. Phys. A 37 (2004), no. 1, 239--250.

\bibitem [Sh01] {Sh} Shen, Chao-Liang Some inverse spectral problems for vectorial Sturm-Liouville equations.
Inverse Problems 17 (2001), no. 5, 1253--1294.

\bibitem [Yur06] {Yu} Yurko, V. Inverse problems for matrix Sturm-Liouville operators.
Russ. J. Math. Phys. 13 (2006), no. 1, 111--118.

\bibitem [ZF71] {ZF} Zaharov, V. E.; Faddeev, L. D. The Korteweg-de Vries equation is a fully integrable Hamiltonian system.
(Russian) Funkcional. Anal. i Prilo\v zen. 5 (1971), no. 4, 18--27.


\end{thebibliography}
\end{document}